\documentclass[11pt]{article}
\usepackage{preamble}
\usepackage{geometry}

\title{Prismatic $F$-crystals and Lubin-Tate \texorpdfstring{$(\varphi_q,\Gamma)$}{(phi\_q, Gamma)}-modules}
\author{Samuel Marks}
\date{\vspace{-3.5ex}}

\begin{document}

\maketitle

\begin{abstract}
	Let $L/\Q_p$ be a finite extension. We introduce \textit{$L$-typical prisms}, a mild generalization of prisms. Following ideas of Bhatt, Scholze, and Wu, we show that certain vector bundles, called Laurent $F$-crystals, on the $L$-typical prismatic site of a formal scheme $X$ over $\Spf\O_L$ are equivalent to $\O_L$-linear local systems on the generic fiber $X_\eta$. We also give comparison theorems for computing the \'etale cohomology of a local system in terms of the cohomology of its corresponding Laurent $F$-crystal. In the case $X = \Spf\O_K$ for $K/L$ a $p$-adic field, we show that this recovers the Kisin-Ren equivalence between Lubin-Tate $(\varphi_q,\Gamma)$-modules and $\O_L$-linear representations of $G_K$ and the results of Kupferer and Venjakob for computing Galois cohomology in terms of Herr complexes of $(\varphi_q,\Gamma)$-modules. We can thus regard Laurent $F$-crystals on the $L$-typical prismatic site as providing a suitable notion of relative $(\varphi_q,\Gamma)$-modules.
\end{abstract}

\section{Introduction}
Let $K/\Q_p$ be a $p$-adic field, let $K_\infty$ be the $p$-adic completion of the infinite cyclotomic extension $K(\zeta_{p^\infty})$, and let $\Gamma_K = \Gal(K_\infty/K)$. In this setting, Fontaine's theory of $(\varphi,\Gamma)$-modules \cite{F} gives an equivalence of categories
\[\Mod_{\A_K}^{\varphi,\Gamma_K,et}\simeq \Mod_{W(K_\infty^\flat)}^{\varphi,\Gamma_K,et}\simeq \Rep_{\Z_p}(G_K)\]
between -- on the representation theoretic side -- the category of finite free $\Z_p$-linear representations of the absolute Galois group $G_K = \Gal(\overline{K}/K)$ and -- on the semi-linear algebraic side -- categories of $(\varphi,\Gamma)$-modules over the \textit{perfect} period ring $W(K_\infty^\flat)$ or a certain \textit{deperfected} period ring $\A_K\subseteq W(K_\infty^\flat)$. Here, the word ``deperfected'' refers to the fact that the imperfect sub-$\F_p$-algebra $\E_K = \A_K/p\subseteq K_\infty^\flat = W(K_\infty^\flat)/p$ becomes $K_\infty^\flat$ under completed perfection.

Following the discussion in \cite[\secsymb0.2]{KL}, we distinguish between two ways one might hope to relativize the theory of $(\varphi,\Gamma)$-modules. First, one might hope for a \textit{geometric} relativization. On the representation theoretic side, this means replacing $\Rep_{\Z_p}(G_K)$ with \'etale local systems $\Loc_{\Z_p}(X_\eta)$ on the generic fiber of a formal scheme $X/\Z_p$. One then hopes to get a corresponding semi-linear algebraic category of objects which can be thought of as $(\varphi,\Gamma)$-modules varying over the base $X$. The most satisfactory candidate here is the \textit{Laurent F-crystals} of \cite{BScrys}. Recall that these are vector bundles $\mc M\in \Vect(X_{\Prism}, \O_\Prism[\tfrac{1}{\mc I}]^\wedge_{(p)})^{\phi = 1}$ over a certain structure sheaf on the prismatic site of $X$ equipped with an isomorphism $\phi^*\mc M\stackrel{\sim}{\rightarrow}\mc M$. Bhatt-Scholze's key theorem is as follows.
\begin{thm}{\em \cite[corollary~3.8]{BScrys}}
	Let $X$ be a bounded formal scheme adic over $\Spf \Z_p$ with adic generic fiber $X_\eta$. Then there is an equivalence $\Vect(X_\Prism, \O_\Prism[\tfrac{1}{\mc I}]^\wedge_{(p)})^{\phi = 1}\simeq \Loc_{\Z_p}(X_\eta)$. 
\end{thm}
\noindent In the case $X = \Spf\O_K$ for $K/\Q_p$ a $p$-adic field, work of Wu \cite{Wu} shows that 
\[\Vect((\O_K)_\Prism, \O_\Prism[\tfrac{1}{\mc I}]^\wedge_{(p)})^{\phi = 1}\simeq \Mod_{\A_K}^{\varphi,\Gamma_K,et}\simeq \Mod_{W(K_\infty^\flat)}^{\varphi,\Gamma_K,et},\] recovering Fontaine's original theory. 

\begin{rmk}
	Due to obstructions related to the fact that Cohen rings can be formed functorially for perfect fields (via the Witt vector construction) but not for arbitrary characteristic $p$ fields, it is significantly easier to give a relative construction of $(\varphi,\Gamma)$-modules over the perfect period ring $W(K_\infty^\flat)$; for example, relative $(\varphi,\Gamma)$-modules over a perfect period sheaf $W(\O_X^\flat)$ are defined in work of Kedlaya and Liu \cite{KL}. In follow-up work, Kedlaya and Liu \cite{KL2} attempt to define satisfactory imperfect period sheaves via an axiomatic approach, but these axioms fail to attain in the important Lubin-Tate case discussed below \cite{SV-decomp}. On the other hand, the Bhatt-Scholze approach to relative $(\varphi,\Gamma)$-modules circumvents this difficulty using the theory of prisms \cite{BS}, which can be viewed as deperfections of perfectoid rings. 
\end{rmk}

Alternatively, one might also want \textit{arithmetic} relativizations of the theory of $(\varphi,\Gamma)$-modules. On the representation theory side, this means replacing the $\Z_p$ in $\Rep_{\Z_p}(G_K)$ with affinoid algebras over $\Z_p$, as in \cite{BC-arith,KPX,Berger-Langlands}. The simplest such case is to study $\Rep_{\O_L}(G_K)$ for $K/L/\Q_p$ a finite subextension. A key goal of this paper is to extend Bhatt-Scholze's prismatic approach to relative $(\varphi,\Gamma)$-modules to this case. We do this by introducing a mild generalization of prisms, which we call $L$-typical prisms, and the $L$-typical prismatic site $X_{\Prism_L}$ of a formal scheme $X/\O_L$. This done, we show the following.
\begin{thm}\label{thm:intro-version-main-result}
	Let $L/\Q_p$ be a finite extension with uniformizer $\pi$, and let $X$ be a bounded formal scheme adic over $\Spf\O_L$ with adic generic fiber $X_\eta$. 
	\begin{enumerate}
		\item[(1)] There is an equivalence of categories
		\[\Vect(X_{\Prism_L}, \O_\Prism[\tfrac{1}{I}]^\wedge_{(\pi)})^{\phi = 1} \simeq \Loc_{\O_L}(X_\eta)\]
		between Laurent $F$-crystals on $X_{\Prism_L}$ and $\O_L$-local systems on $X_\eta$.
		\item[(2)] If $\mc M\in \Vect(X_{\Prism_L}, \O_\Prism[\tfrac{1}{I}]^\wedge_{(\pi)})^{\phi = 1}$ and $T\in \Loc_{\O_L}(X_\eta)$ correspond under the equivalence above, then there is an isomorphism
		\[R\Gamma(X_{\Prism_L}, \mc M)^{\phi = 1} \cong R\Gamma(X_{\eta,et}, T).\]
	\end{enumerate}
\end{thm}
Note that this theorem comes with an \'etale comparison generalizing \cite[theorem~1.9(i)]{Guo}, itself a generalization of the Bhatt-Scholze \'etale comparison \cite[theorem~1.8(4)]{BS}. Here and throughout the paper, if $E$ is a complex in a derived category with an endomorphism $\phi$, then $E^{\phi = 1} := \mathrm{Cone}(\phi - \mathrm{id})[-1]$ is the mapping cocone of $\phi - \mathrm{id}$. 

Before going on, we say a few words about $L$-typical prisms, which were independently defined by Ito and called ``$\O_L$-prisms'' in his concurrent work \cite{Ito}. The category of $L$-typical prisms is a mild generalization of the category of prisms, arising by replacing $\delta$-rings with what we call $\delta_L$-algebras. In the same way that $p$-complete $\delta$-rings relate to $\Z_p$-algebras with a lift of Frobenius, $\delta_L$-algebras relate to $\O_L$-algebras with a lift of $q$-Frobenius. And just as the category of prisms has a subcategory of perfect prisms, which is equivalent to the category of (integral) perfectoid rings (this is what we mean when we say that prisms can be viewed as ``deperfections of perfectoid rings''), we will show the following.
\begin{thm}\label{thm:intro-version-perfect-prism-pfctd-OL-alg}
	Let $L/\Q_p$ be a finite extension. The categories of $L$-typical prisms and perfectoid $\O_L$-algebras (i.e. integral perfectoid rings which are also $\O_L$-algebras) are equivalent.
\end{thm}
\begin{rmk}
	The notion of $\delta_L$-algebras defined here coincides with Borger's notion of a $\pi$-typical $\Lambda_{\O_L}$-ring \cite{Borger}. More generally, following a suggestion of Kisin, the author suspected that Borger's $\Lambda$-rings were the right formalism for arithmetically relativizing $(\varphi,\Gamma)$-modules in general. We hope that this work -- which carries out this relativization in the simplest case beyond $\Z_p$-coefficients -- provides evidence that the same techniques will be useful more generally.
\end{rmk}

Fix now a Lubin-Tate formal $\O_L$-module $\G$ corresponding to the uniformizer $\pi$ of $\O_L$. If $K/L$ is a $p$-adic field, then we let $K_\infty$ be the $p$-adic completion of the infinite extension $K(\G[\pi^\infty])$ formed by adjoining the $\pi$-power torsion points of $\G$. In this case, one can use the periods of $\G$ to construct an element $\omega\in W(K_\infty^\flat)\otimes_{W(\F_q)}\O_L$ and a period ring $\A_K\subseteq W(K_\infty^\flat)\otimes_{W(\F_q)}\O_L$ (different in general from the ring $\A_K$ discussed above, but coinciding in the cyclotomic case $\G = \mu_{p^\infty}$). One also gets a category $\Mod_{\A_K}^{\varphi_q,\Gamma_K}$ of \textit{Lubin-Tate $(\varphi_q,\Gamma)$-modules}, first studied by Kisin and Ren \cite{KR} following ideas of Fontaine, and recently a subject of significant interest in the context of explicit reciprocity laws, $p$-adic local Langlands, and Iwasawa theory \cite{Berger-Iwasawa,Berger-Langlands,Schneider-Coates,SV20,Fourquaux_2013}. 

In \secsymb\ref{ssec:prism-maps} we give general constructions for producing interesting subprisms of a perfect $L$-typical prisms. When applied with inputs derived from periods of $\G$ and the perfect $L$-typical prism $(\Ainf(\O_{K_\infty}),\ker\theta)$ corresponding via theorem~\ref{thm:intro-version-perfect-prism-pfctd-OL-alg} to the perfectoid $\O_L$-algebra $\O_{K_\infty}$, we show that this construction produces a prism $(\A_K^+,(q_n(\omega)))$ with $\A_K = \A_K^+[\tfrac{1}{q_n(\omega)}]^\wedge_{(\pi)}$. This period ring interestingly depends on the Lubin-Tate formal group $\G$; for example, we construct a prismatic logarithm map $T\G\rightarrow \A_L^+\{1\}$ to the Breuil-Kisin twist, as in \cite{BL}. Using the prism $(\A_K^+,(q_n(\omega)))$, we show that theorem~\ref{thm:intro-version-main-result} recovers both the Kisin-Ren equivalence $\Mod_{\A_K}^{\varphi_q,\Gamma_K,et}\simeq \Rep_{\O_L}(G_K)$ as well as the computation of Galois cohomology in terms of $\varphi$-Herr complexes from \cite{Venjacob-Herr}.

\begin{thm}\label{thm:intro-version-kisin-ren}
	Let $L/\Q_p$ be a finite extension with uniformer $\pi$, and let $K/L$ be a $p$-adic field.
	\begin{enumerate}
		\item[(1)] There are equivalences of categories
		\begin{align*}
		 	\Mod_{\A_K}^{\varphi_q,et}\simeq \Mod_{W_L(K_\infty^\flat)}^{\varphi_q,et}&\simeq \Vect((\O_{K_\infty})_{\Prism_L}, \O_\Prism[\tfrac{1}{\mc I}]^\wedge_{(\pi)})^{\phi = 1}\simeq \Rep_{\O_L}(G_{K_\infty}) \\
			\Mod_{\A_K}^{\varphi_q,\Gamma_K,et}\simeq \Mod_{W_L(K_\infty^\flat)}^{\varphi_q,\Gamma_K,et}&\simeq \Vect((\O_{K})_{\Prism_L}, \O_\Prism[\tfrac{1}{\mc I}]^\wedge_{(\pi)})^{\phi = 1}\simeq \Rep_{\O_L}(G_{K}).
		 \end{align*}
		(Here $W_L(K_\infty^\flat) = \Ainf(\O_{K_\infty})[\tfrac{1}{\ker\theta}]^\wedge_{(\pi)}$ is the period ring corresponding to the perfect $L$-typical prism $(\Ainf(\O_{K_\infty}),\ker\theta)$.)
		\item[(2)] If $M\in \Mod_{\A_K}^{\varphi_q,et}$ corresponds to $T\in \Rep_{\O_L}(G_{K_\infty})$ under the above equivalence, then 
		\[R\Gamma(K_{\infty,et}, T) \cong \left(M\stackrel{\phi - 1}{\longrightarrow} M\right)\]
		where the complex on the right is concentrated in degrees $0$ and $1$. 
		\item[(3)] If $M\in \Mod_{\A_K}^{\varphi_q,\Gamma_K,et}$ corresponds to $T\in \Rep_{\O_L}(G_K)$, then 
		\[R\Gamma(K_{et},T)\cong C_{\mathrm{cont}}^\bullet(\Gamma_K,M)^{\phi = 1}\] 
		where $C_{\mathrm{cont}}^{\bullet}(\Gamma_K, M)$ denotes the continuous cochain complex of $\Gamma_K$ with values in $M$.
	\end{enumerate}
\end{thm}

\subsection{Explicit reciprocity laws and Iwasawa theory}
A key motivation for this work is explicit reciprocity laws in Iwasawa theory. Let $K_n = \Q_p(\zeta_{p^n})$ and $K = \Q_p$. In the most classical case, Iwasawa's explicit reciprocity law \cite{Iwasawa} computes, for a system $u = (u_n)_n\in \plim K_n^\times$ of $p$-power compatible units and $m\ge 1$, the image of $u$ under the composition
\[\lambda_m:\plim K_n^\times\stackrel{\kappa}{\longrightarrow} \plim H^1(K_n, \Z_p(1))\cong \plim H^1(K_n,\Z_p(k))\stackrel{\mathrm{Tr}_{K_n/K_m}}{\rightarrow} H^1(K_m,\Z_p)\stackrel{\exp^*}{\longrightarrow} K_m\]
where $\kappa$ is the Kummer map, the isomorphism is a Soul\'e twist\footnote{Concretely, using the isomorphism $\plim H^1(K_n, \Z_p(1))\cong H^1(K, \Z_p\ps{\Gamma_K}\otimes_{\Z_p}\Z_p(1))$, the Soul\'e twist arises from the isomorphism $\Z_p\ps{\Gamma_K}\longrightarrow \Z_p\ps{\Gamma_K}\otimes_{\Z_p}\Z_p(1)$ of $G_K$-modules given by $\gamma\mapsto \gamma\otimes \gamma e$ corresponding to a choice of basis $e$ of $\Z_p(1)$.}, and $\exp^*$ is the Bloch-Kato dual exponential map \cite[II.1.2]{Kato1993}. Explicitly, 
\[\lambda_m(u) = p^{-m}u_m(\dlog\theta_u)(u_m - 1)\]
where $\theta_u\in \Z_p\ps{T}^\times$ is the Coleman power series for $u$ and $\dlog\theta = \frac{\theta'(T)}{\theta(T)}$. The Iwasawa cohomology group $H^1_{Iw}(K_\infty/K,\Z_p(1)):=\plim H^1(K_n/K,\Z_p(1))$ is important, in part, because its contains as an element the Euler system of cyclotomic units. This formula for $\lambda_m$ thereby allows one to relate this Euler system to zeta values.

More generally, let $L/\Q_p$ be a finite extension with uniformizer $\pi$, let $\G$ be a Lubin-Tate formal $\O_L$-module corresponding to $\pi$, let $L_n = L(\G[\pi^n])$, and let $T\G\in \Rep_{\O_L}(G_L)$ be the Tate module of $\G$. Then for each $m\ge 1$ and $k\in\Z$ there is a map
\[\lambda_{m,k}:\plim L_n^\times\stackrel{\kappa}{\rightarrow}H^1_{Iw}(L_\infty/L, \Z_p(1))\cong H^1_{Iw}(L_\infty/L, T\G^{\otimes -k}(1))\stackrel{\mathrm{Tr}}{\rightarrow} H^1(L_m, T\G^{\otimes -r}(1))\stackrel{\exp^*}{\longrightarrow} L_m t_\G^{k}t_{cycl}^{-1}\]
where $t_\G\in D_{dR}(T\G^{\otimes -1})$ and $t_{cycl}\in D_{dR}(\O_L(-1))$ are the usual de Rham periods. Then work of Bloch and Kato \cite{BK} gives the explicit reciprocity law
\[\lambda_{m,k}(u) = \frac{1}{k!}\pi^{-mk}(\partial_\G^k \log\theta_u)(u_m)t_\G^kt_{cycl}^{-1}\]
for $k\ge 1$, where $\theta_u\in\O_L\ps{T}$ is again a Coleman power series and $\partial_\G(f(T)) := \frac{1}{g(T)}f'(T)$ with $g(T)dT$ being the invariant differential for $\G$.

Intuitively speaking, for a fixed $k\ge 1$, the above explicit reciprocity law for $\lambda_{m,k}$ extracts information from the system $(u_n)_{n\ge 1}$ related to the special value of a $p$-adic $L$-function at $s = k$. On the other hand, work of Perrin-Riou, Colmez, and Cherbonnier \cite{PR,CC} in the cyclotomic case $\G = \mu_{p^\infty}$ and Schneider and Venjakob \cite{Schneider-Coates} in the general case shows how to interpolate all of the above ``little'' explicit reciprocity laws into one ``big'' explicit reciprocity law which sees the entire $p$-adic $L$-function at once. More precisely, if $M\in \Mod_{\A_L}^{\varphi_q,et}$ corresponds to $T\G\in \Rep_{\O_L}(G_L)$ under theorem~\ref{thm:intro-version-kisin-ren}, then there is a big dual exponential map \cite[\secsymb5]{Schneider-Coates}
\[\mathrm{Exp}^*: H^1_{Iw}(L_\infty/L,\O_L(1))\stackrel{\sim}{\longrightarrow} M^{\psi = 1}\]
where $\psi$ is a certain endomorphism of $M$. Moreover, we have $M \cong \Omega^1_{\O_\G/\O_L}\cong \Omega^1_{\O_L\ps{T}/\O_L}$ and the big explicit reciprocity law
\[(\mathrm{Exp}^*\circ \kappa)(u) = \dlog \theta_u.\]
Intuively, this shows how to relate a $p$-adic $L$-function corresponding to a system $(u_n)_n$ of units to a function $\theta_u\in \O_\G\cong \A_L\cong \O_L\ps{T}$ on the Lubin-Tate group $\G$.

Two ingredients were essential for the above big explicit reciprocity law to be formulated and proved. First, there is a map $\O_G\rightarrow \A_L$ from the ring of functions on $\G$ to the period ring for the $\varphi$-modules. Second, the period ring $\A_L$ is \textit{imperfect}; indeed, the corresponding perfect period ring $W_L(L_\infty^\flat)$ has $\Omega^1_{W_L(L_\infty^\flat)/\O_L} = 0$, presenting a fundamental obstruction to a big explicit reciprocity law like the one above. Moreover, in \cite{Schneider-Coates}, $\psi$ is shown to be related to the endomorphism $\phi$ of $\A_L$ via Pontryagin duality using an argument that makes use of local Tate duality and a \textit{residue} pairing
\[\A_L\otimes_{\A_L}\Omega^1_{\A_L/\O_L}\stackrel{\mathrm{res}}{\longrightarrow}\O_L,\]
which suggests that $\A_L$ being not too much larger than $\O_\G$ is key.

In settings beyond the case of Lubin-Tate formal groups, there are families of little explicit reciprocity laws which lack big explicit reciprocity laws. For instance Kato's generalized explicit reciprocity law \cite{Kato-GER}, a key technical ingredient to Kato's work \cite{Kato-euler} on Iwasawa main conjectures for modular forms, is used to relate special values of $L$-functions with special values of derivatives of logarithms of \textit{Siegel units}, which are certain functions on the $p$-divisible group of an elliptic curve. 

The author suspects that the path forward in formulating and proving big explicit reciprocity laws in this setting involves constructing certain imperfect prisms $(A,I)$ over (the ordinary locus of) a modular curve $X$ such that the $p$-divisible group $\mathcal{E}[p^\infty]$ of the universal elliptic curve $\mathcal E\rightarrow X$ has a map $\O_{\mathcal E[p^\infty]}\rightarrow A$. Some partial progress is presented in example~\ref{example:elliptic}: given an ordinary elliptic curve over a $p$-complete ring $R$ equipped with a compatible system of sections $\Spf R_n\rightarrow \ker F^n$ of the subgroups $\ker F^n$ over \'etale $R$-algebras, the general constructions given in \secsymb~\ref{ssec:prism-maps} produce a map $\O_{\varinjlim \ker F^n}\rightarrow W((\varinjlim R_n)^\flat)$. (If $(\varinjlim R_n)^\wedge_{(p)}$ is perfectoid, then $(W((\varinjlim R_n)^\flat), \ker\theta)\in R_{\Prism}$ is a perfect prism.)

\subsection{Overview of the proofs}
We briefly outline the key ideas in the proofs of theorems~\ref{thm:intro-version-main-result} and \ref{thm:intro-version-kisin-ren}. When $X = \Spf R$ for a perfectoid $\O_L$-algebra $R$, $\Vect(R_{\Prism_L}, \O_\Prism[\tfrac{1}{\mc I}]^\wedge_{(\pi)})^{\phi = 1}\simeq \Mod_{W(R[\tfrac1\pi]^\flat)}^{\varphi_q,et}$, and theorem~\ref{thm:intro-version-main-result} is shown via standard arguments (due originally to Katz and Fontaine \cite{K,F}) for relating \'etale $\varphi$-modules and local systems. Theorem~\ref{thm:intro-version-main-result} is then shown in general via a descent argument from the perfectoid case. This crucially relies on the fact that there is a perfection functor $(A,I)\mapsto (A,I)_{\perf}$ which induces an equivalence on the corresponding categories of \'etale $\varphi_q$-modules.
\begin{thm}\label{thm:intro-version-perfection-base-change}
		(c.f. \cite[theorem~4.6]{Wu} for the $\Q_p$-typical case). Let $(A,I)$ be a bounded $L$-typical prism with perfection $(A_{\perf}, IA_{\perf})$. Then base change induces an equivalence
		\begin{align*}
			\Mod_{(A,I)}^{\phi,et}&\stackrel{\sim}{\longrightarrow}\Mod_{(A,I)_{\perf}}^{\phi,et} \\
			M&\mapsto M\otimes_{A[\tfrac{1}{I}]^\wedge_{(\pi)}}A_{\perf}[\tfrac{1}{I}]^\wedge_{(\pi)}
		\end{align*}
		between the categories of \'etale $\varphi_q$-modules over $(A,I)$ and $(A,I)_{\perf}$. 
\end{thm}

The $X = \Spf\O_{K_\infty}$ part of theorem~\ref{thm:intro-version-kisin-ren} follows nearly immediately from theorem~\ref{thm:intro-version-main-result}. Intuitively, one would like to conclude the $X = \Spf\O_K$ part by descending along $Y = \Spf\O_{K_\infty}\rightarrow X = \Spf\O_K$ and picking up a semilinear action of $\Gamma_K = \Gal(K_\infty/K)$. However, instead of using this angle of attack, we will use a more delicate descent argument along the \v{C}ech nerve $(W_L(\O_{K_\infty}^\flat), \ker\theta)^\bullet$ in the perfect prismatic site $(\O_K)_{\Prism_L}^{\perf}$. This argument allows us to recover a Laurent $F$-crystal $\mc M$ over $(\O_K)_{\Prism_L}$ from the data of $M = \mc M(W_L(\O_{K_\infty}^\flat), \ker\theta)$ and a semilinear action of $\Aut_{(\O_K)_{\Prism_L}}(W_L(\O_{K_\infty}^\flat), \ker\theta)\cong \Gamma_K$, and to compute $R\Gamma((\O_K)_{\Prism_L}, \mc M) \cong C_{\mathrm{cont}}^\bullet(\Gamma_K, M)$. 

\subsection{Structure of the paper}

In \secsymb\ref{sec:delta-algebras} we introduce $\delta_L$-algebras, review ramified Witt vectors, and develop basic results about distinguished elements and perfect $\delta_L$-algebras. In \secsymb\ref{sec:L-typical-prisms} we then introduce $L$-typical prisms, with perfectoid $\O_L$-algebras, the proof of theorem~\ref{thm:intro-version-perfect-prism-pfctd-OL-alg}, and the perfection functor appearing in \secsymb\ref{ssec:perfect-prisms-perfectoid-algebras}. In \secsymb\ref{ssec:prism-maps} we describe two general constructions which -- given an $L$-typical prism $(A,I)$, a perfectoid $\O_L$-algebra $R$, and a $\phi$-compatible system of maps $(\iota_n:A\rightarrow R)_n$ -- produce a map $(A,I)\rightarrow (\Ainf(R),\ker\theta)$ to the perfect $L$-typical prism corresponding to $R$; the example~\ref{example:elliptic} discussed above, involving constructing a map from a sub-$p$-divisible group of the $p$-divisible group of an elliptic curve to $W((\varinjlim R_n)^\flat)$, is also sited here.

Starting in \secsymb\ref{sec:LT-phi-G-modules}, we will take $\G$ to be a Lubin-Tate formal $\O_L$-module corresponding to a uniformizer $\pi$ of $L$. We explain in \secsymb\ref{ssec:S_K} how to equip $\O_\G\cong \O_L\ps{T}$ with ideals $(q_n(T))$ which turn it into an $L$-typical prism; furthermore, the constructions from \secsymb\ref{ssec:prism-maps} allow us to, given a choice of basis $e$ for the rank one $\O_L$-module $T\G$, produce an embedding $(\O_\G, (q_n(T)))\hookrightarrow (W_L(\O_{L_\infty}^\flat), \ker\theta)$ into a perfect prism. Given a $p$-adic field $K/L$, we extend this construction in \secsymb\ref{ssec:A_K} to give a prism $(\A_K^+, (q_n(\omega)))\in (\O_K)_{\Prism_L}$ with perfection $(W_L(\O_{K_\infty}^\flat), \ker\theta)$. In \secsymb~\ref{ssec:phi_q-Gamma-modules} we review the basics of the theory of Lubin-Tate $(\varphi_q,\Gamma)$-modules and the $\Gamma_K$-action on $\A_K$. Then \secsymb~\ref{ssec:prismatic-logarithm} contains discussion of the prismatic logarithm for $\G$; we included this section because we believed the construction was interesting, but it plays no further role in this paper.

Finally, \secsymb\ref{sec:F-crystals} is the technical heart of the paper. In \secsymb\ref{ssec:phi-modules-over-prisms} we define $\varphi_q$-modules over $L$-typical prisms and prove theorem~\ref{thm:intro-version-perfection-base-change}. Then \secsymb\ref{ssec:laurent-F-crystals} defines Laurent $F$-crystals and proves theorem~\ref{thm:intro-version-main-result}, with theorem~\ref{thm:kisin-ren-equiv} following in \secsymb\ref{ssec:phi-G-modules-F-crystals}.

\subsection{Acknowledgements}

This work would not have been possible without the support, guidance, and frequent prophetic suggestions of Mark Kisin. I also thank Alexander Petrov for help with various aspects of the prismatic theory, and Daniel Li-Heurta for help with v-descent results for diamonds. Finally, an earlier version of this document contained errors which were kindly pointed out by Kazuhiro Ito, including that my original statement of theorem~\ref{thm:intro-version-perfect-prism-pfctd-OL-alg} was incorrect; I thank Dr. Ito for identifying these errors and for helping me arrive at a proof for the corrected theorem~\ref{thm:intro-version-perfect-prism-pfctd-OL-alg}.

\section{\texorpdfstring{$\delta_L$}{delta\_L}-algebras and ramified Witt vectors}\label{sec:delta-algebras}
Recall that a $\delta$-ring is a ring $A$ together with a map $\delta:A\rightarrow A$ of sets satisfying certain properties which guarantee that
\begin{align*}
	\phi : &A\longrightarrow A\\
	&x\mapsto x^p + p\delta(x)
\end{align*}
is a ring homomorphism lifting the Frobenius endomorphism $x\mapsto x^p$ of $A/p$. In this section, we will recall a mild generalization of the theory of $\delta$-rings which applies in the following context.

Let $L/\Q_p$ be a finite extension with ring of integers $\O_L$, uniformizer $\pi$, and residue field $\O_L/\pi$ of size $q$. Then a $\delta_L$-algebra will be an $\O_L$-algebra $A$ equipped with a map $\delta_L:A\rightarrow A$ of sets satisfying certain properties which guarantee that
\[\phi(x) = x^q + \pi\delta_L(x)\]
is a ring homomorphism lifting the $q$-Frobenius $\varphi_q(x) = x^q$ of $A/\pi$. 

\begin{rmk}
	By a theorem of Wilkerson \cite{Wilk}, $\delta$-rings are the same as $p$-typical $\lambda$-rings, a notion generalized by the $\Lambda$-rings of Borger \cite{Borger}. The results of this section are obtained as special cases of Borger's theory of $\Lambda$-rings over $\O_L$ in the $\pi$-typical setting.
\end{rmk}

\subsection{Basic theory}\label{ssec:delta-algebras-basic-theory}

\begin{defn}~
	\begin{enumerate}
		\item[(1)] A \textit{$\delta_L$-algebra} is an $\O_L$-algebra $A$ equipped with a map $\delta_L:A\rightarrow A$ of sets satisfying the identities
		\begin{align}
			\delta_L(\alpha) &= \frac{\alpha - \alpha^q}{\pi} &\text{for }\alpha\in \O_L \notag \\
			\delta_L(xy) &= \delta_L(x)y^q + x^q\delta_L(y) + \pi\delta_L(x)\delta_L(y) &\text{for }x,y\in A \notag \\
			\delta_L(x + y) &= \delta_L(x) + \delta_L(y) + \frac{x^q + y^q - (x + y)^q}{\pi} &\text{for } x,y\in A\label{delta-sum-identity}
		\end{align}
		where in (\ref{delta-sum-identity}) the expression $\frac{x^q + y^q - (x + y)^q}{\pi}$ is shorthand for
		\[-\sum_{i = 1}^{q - 1} \frac{1}{\pi}\binom{q}{i}x^iy^{q - i}\]
		which makes sense even when $A$ has $\pi$-torsion. If $A$ is an $\O_L$-algebra then by a \textit{$\delta_L$-structure} on $A$ we mean a choice of map $\delta_L:A\rightarrow A$ as above making $A$ into a $\delta_L$-algebra.

		\item[(2)] There is an evident category $\Alg_{\delta_L}$ of $\delta_L$-algebras, with maps being $\O_L$-algebra maps which commute with the $\delta_L$-structures. 


		\item[(3)] If $A$ is a $\delta_L$-algebra, then we have a map $\phi_{A,\delta_L}:A\rightarrow A$ given by $\phi_{A,\delta_L}(x) = x^q + \pi\delta_L(x)$ which lifts the $q$-Frobenius $\varphi_q$ on $A/\pi$. Using the assumed identities on $\delta_L$, one verifies that $\phi_{A,\delta_L}$ is an $\O_L$-algebra homomorphism. Usually $A$ and $\delta_L$ will be clear from context and we will simply write $\phi_A$ or $\phi$ for $\phi_{A,\delta_L}$.

	\end{enumerate}

\end{defn}

\begin{rmk}~
	\begin{enumerate}
		\item[(1)] \label{rmk:d_pi-q-frob} If $A$ is a $\pi$-torsion-free $\O_L$-algebra and $\phi$ is an endomorphism lifting $\varphi_q$, then we obtain a $\delta_L$-structure on $A$ by 
		\[\delta_L(x) = \frac{\phi(x) - x^q}{\pi}.\]
		This is easily seen to give a one-to-one correspondence between $\delta_L$-structures on $A$ and lifts of $\varphi_q$ to $A$. When $A$ has $\pi$-torsion, having a $\delta_L$-structure is stronger than having a lift of $\varphi_q$.
		\item[(2)] \label{rmk:q-derivation} Taking the defining relations of a $\delta_L$-structure modulo $\pi$, we see that $\delta_L$-structure on an $\O_L$-algebra $A$ induces an $\F_q$-module map $\delta_L : A/\pi\rightarrow A/\pi$ such that $\delta_L(\alpha) = 0$ for $\alpha\in \F_q$ and we have the analogue
		\[\delta_L(xy) = x^q\delta_L(y) + y^q\delta_L(x)\]
		of the Leibnitz rule.
		\item[(3)] The properties defining the map $\delta_L$ evidently depend on the choice of uniformizer $\pi$, so one might worry that $\delta_L$ algebra structures on an $\O_L$-algebra $A$ might depend on the choice of $\pi$ as well. Fortunately, there is no essential dependence: if $A$ has a $\delta_L$-structure with respect to $\pi$ and $\pi' = u\pi$ for $u\in \O_L^\times$ is another uniformizer, then $\alpha\mapsto u^{-1}\delta_L(\alpha)$ is a $\delta_L$-structure with respect to $\pi'$.
		\item[(4)] \label{rmk:lim-colim-closed} See \cite[remark~2.2.7]{Ito} (generalizing \cite[remark~2.4]{BS}) for an alternative characterization of $\delta_L$-structures on $A$ in terms of $\O_L$-algebra sections of the length $2$ ramified Witt vectors $W_{L,2}(A)$. In particular, this characterization immediately implies that the category of $\delta_L$-algebras admits all limits and colimits, and that they commute with the forgetful functor to $\O_L$-algebras. 
		\item[(5)] \label{rmk:I-adic-closed} The category of $\delta_L$-algebras is also closed with respect to classical $I$-adic completion with respect to an ideal $I\subseteq A$ containing $\pi$ (cf. \cite[lemma~2.2.10]{Ito} or the proof of \cite[lemma~2.17]{BS}).
	\end{enumerate}
\end{rmk}

A key fact about $\delta_L$-algebras is that the forgetful functor $\Alg_{\delta_L}\rightarrow \Alg_{\O_L}$ has a right adjoint $W_L$, which is identified with Hazewinkel's ramified Witt vector functor \cite{Haz}. Explicitly, for $n\ge 0$, let 
\[ w_n(X_0, \dots, X_n) = X_0^{q^n} + \pi X_1^{q^{n - 1}} + \dots + \pi^{n - 1}X_{n - 1}^q + \pi^nX_n\in \O_L[X_0, \dots, X_n]\subseteq \O_L[X_0, X_1, \dots] \]
be the $n$th ghost component polynomial. For any $\O_L$-algebra $R$, let $W_L(R) = R^\N$ as sets, and let
\begin{align*}
	w_R: W_L(R)&\longrightarrow R^\N \\
	x = (x_0, x_1, \dots) &\mapsto (w_0(x), w_1(x), \dots)
\end{align*}
be the ghost component map. Since $w_R$ is a bijection when $R$ is $\pi$-torsion free and any $\O_L$-algebra is a quotient of a free $\O_L$-algebra, there is a unique choice of $\O_L$-algebra structure on $W_L(R)$ such that $w_R$ is a map of $\O_L$-algebras \textit{and} $W_L$ is a functor $\Alg_{\O_L}\rightarrow \Alg_{\O_L}$; equip $W_L(R)$ with this $\O_L$-algebra structure. One also checks that the projection map $W_L(R) = R^\N\twoheadrightarrow R$ onto the first factor is an $\O_L$-algebra homomorphism.


The above paragraph explains the $\O_L$-algebra structure on $W_L(R)$; we now explain the $\delta_L$-structure. In the case that $R$ is $\pi$-torsion-free, $W_L(R)$ is $\pi$-torsion-free as well, so giving a $\delta_L$-structure is the same as giving a lift of $q$-Frobenius. This is provided by the canonical Witt vector Frobenius.
\begin{prop}
	If $R$ is an $\O_L$-algebra, then there are endomorphisms $F_R$ and $V_R$ of $W_L(R)$, natural in $R$, such that for $x,y\in W_L(R)$ we have
	\begin{align*}
		F_R(x) &\equiv x^q\mod{\pi W_L(R)}, \\
		F_R(V_R(x)) &= \pi x, \\
		V_R(xF_R(y)) &= V_R(x)y, \\
	\end{align*}
	and the diagrams
	\begin{equation}\label{frobenius-ghost}
		\begin{tikzcd}[sep=large]
			W_L(R)\ar{r}{w_R} \ar{d}{F_R} & R^\N\ar{d}{(w_0, w_1, \dots)\mapsto (w_1, w_2, \dots)} &&& W_L(R)\ar{r}{w_R}\ar{d}{V_R} & R^\N\ar{d}{(w_0,w_1,\dots)\mapsto (0, \pi w_0, \pi w_1,\dots)} \\
			W_L(R)\ar{r}{w_R} & R^\N &&& W_L(R)\ar{r}{w_R} & R^\N
		\end{tikzcd}
	\end{equation}
	commute.
\end{prop}
\begin{proof}
	This uses the same arguments as for $p$-typical Witt vectors; see \cite[pg.~14]{SchnBook} for details.
\end{proof}

In fact, $W_L(R)$ has a $\delta_L$-structure even when $R$ is not $\pi$-torsion-free.

\begin{lem}~\label{lem:W-delta_pi-functor}
$W_L$ extends to a functor $\Alg_{\O_L}\rightarrow \Alg_{\delta_L}$ which is right adjoint to the forgetful functor. Explicitly, this means that if $A$ is a $\delta_L$-algebra then any $\O_L$-algebra map $A\rightarrow R$ lifts to a unique $\delta_L$-algebra map $A\rightarrow W_L(R)$ making the following diagram commute. 
	\begin{center}
		\begin{tikzcd}
			A\ar{dr}\ar[dashed]{r} & W_L(R)\ar[twoheadrightarrow]{d} \\
			& R
		\end{tikzcd}
	\end{center}
\end{lem}
\begin{proof}
	See \cite{Borger}.
\end{proof}

We will make use of two distinct sections of $W_L(R)\twoheadrightarrow R$. One is the usual \textit{Teichm\"uller} lift $r\mapsto [r]$, a multiplicative section which exists for any $\O_L$-algebra $R$. In the case that $R$ also has a $\delta_L$-structure, another section exists which is moreover a $\delta_L$-algebra map.

\begin{prop}~\label{prop:W_pi-sections}
	\begin{enumerate}
		\item[(1)] If $R$ is any $\delta_L$-algebra, then there is a unique map $s_R:R\rightarrow W_L(R)$ of $\delta_L$-algebras which is a section of $W_L(R)\twoheadrightarrow R$. It satisfies $w_n(s_R(\alpha)) = \phi_R^n(\alpha)$ for all $n\ge 0$ and $\alpha\in R$.
		\item[(2)] If $R$ is a $\O_L$-algebra, then the map
		\begin{align*}
			[-] : R&\longrightarrow W_L(R) \\
			r&\mapsto (r, 0, 0,\dots)
		\end{align*}
		is a multiplicative section of $W_L(R)\twoheadrightarrow R$. 
	\end{enumerate}
\end{prop}

Note that if $R$ is $\pi$-torsion-free, then the formula in (1) uniquely determines the map $s_R$.

\begin{proof}
For part (1), $s_R$ is the unit of the adjunction from lemma~\ref{lem:W-delta_pi-functor} (i.e. apply the lemma to $\id: R\rightarrow R$). The formula for $w_n(s_R(\alpha))$ follows from the left diagram in (\ref{frobenius-ghost}) and the defining property of $s_R$ as 
	\[w_n(s_R(\alpha)) = w_0(F_{W_L(R)}^ns_R(\alpha)) = w_0(s_R(\phi_R^n\alpha)) = \phi_R^n\alpha.\]
Part (2) is clear, as one only needs to check that formula given defines a multiplicative map. But let us explain the relationship to part (1): let $R^\circ$ denote $R$ viewed as a multiplicative monoid. Then the free $\O_L$-algebra $\O_L[R^\circ]$ has a lift of $q$-Frobenius induced by $r\mapsto r^q$. Thus applying lemma~\ref{lem:W-delta_pi-functor} to the canonical map $\O_L[R^\circ]\rightarrow R$ gives a $\delta_L$-algebra map $\O_L[R^\circ]\rightarrow W_L(R)$, and the Teichm\"uller map is the composite
\[R^\circ\rightarrow \O_L[R^\circ]\rightarrow W_L(R).\]
To get the formula $[r] = (r, 0, 0, \dots)$, one uses the same reasoning as in part (1) to show that this formula holds when $R$ is $\pi$-torsion-free, from which it follows in general.
\end{proof}


\subsection{Distinguished elements and perfect \texorpdfstring{$\delta_L$}{delta\_L}-algebras}

This section develops results about distinguished elements and perfect $\delta_L$-algebras analogous to those in \cite[\secsymb2.3,\secsymb2.4]{BS}.

\begin{defn}
	Let $A$ be a $\delta_L$-algebra. An element $d\in A$ is \textit{distinguished} if $\delta_L(d)$ is a unit of $A$.
\end{defn}

\begin{rmk}~\label{rmk:pi-distinguished}
	\begin{enumerate}
		\item[(1)] As $\delta_L(\pi) = 1 - \pi^{q - 1}$, we have that $\pi$ is distinguished in any $\delta_L$-algebra.
		\item[(2)] The significance of distinguished elements is that if $(A,I)$ is a $L$-typical prism (to be introduced in \secsymb\ref{sec:L-typical-prisms}), then $I$ is locally generated by distinguished elements (see condition (iii) in the following lemma). As such, we are interested in the case that $A$ is $d$-adically complete; more generally we will assume that $d\in \mathrm{Rad}(A)$ is in the Jacobson radical of $A$.
	\end{enumerate}
\end{rmk}

\begin{lem}\label{lem:dist-characterization}
	Let $A$ be a $\delta_L$-algebra, and let $d\in \mathrm{Rad}(A)$. The following are equivalent:
	\begin{enumerate}
		\item[(i)] $d$ is distinguished.
		\item[(ii)] The ideal $(d)$ contains a distinguished element.
		\item[(iii)] $\pi\in (d^q,\phi(d))$.
		\item[(iv)] $\pi\in (d,\phi(d))$.
	\end{enumerate}
\end{lem}
\begin{proof}
	Clearly (i)$\implies$(ii). Conversely, suppose we have $d' = \alpha d$ for some $\alpha,d'\in A$ with $d'$ distinguished. Applying $\delta_L$ and working mod $(\pi, d)$ (using remark~\ref{rmk:q-derivation}(2) to simplify) we have
	\[\delta_L(d')\equiv \alpha^q\delta_L(d)\pmod{(\pi, d)},\]
	which shows that $\delta_L(d)$ is a unit in $A/(\pi, d)$. As $\pi, d\in \mathrm{Rad}(A)$, we have that $\delta_L(d)\in A^\times$ as well.

	We now show (i)$\implies$(iii)$\implies$(iv)$\implies$(i). The first implication follows directly from the formula $\phi(d) = d^q + \pi \delta_L(d)$, and the second implication is clear. For the last implication, suppose that $\pi = \alpha d + \beta \phi(d)$ for some $\alpha,\beta\in A$. Applying $\delta_L$ to this formula and working mod $(\pi, d)$ we get
	\[\delta_L(\pi) \equiv \delta_L(d)(\alpha^q + \beta^q\delta_L(d)^{q - 1})\pmod{(\pi, d)}.\]
	Then since $\pi$ is distinguished in any $\delta_L$-algebra, we conclude that $\delta_L(d)$ is a unit in $A/(\pi, d)$ and thus in $A$ as well.
\end{proof}

\begin{defn}
	A $\delta_L$-algebra $A$ is \textit{perfect} if $\phi_A$ is an isomorphism. 
\end{defn}

\begin{lem}{(See \cite[Lemma~2.28]{BS}.)}\label{lem:perfect-implies-torsionfree}
	Let $A$ be a $\delta_L$-algebra. Then if $\alpha\in A$ is $\pi$-torsion, we have $\phi(\alpha) = 0$. In particular, if $A$ is perfect then $A$ is $\pi$-torsion free.
\end{lem}
\begin{proof}
	Applying $\delta_L$ to $\pi \alpha = 0$ gives
	\[0 = \pi^q\delta_L(\alpha) + \delta_L(\pi)\alpha^q + \pi \delta_L(\pi)\delta_L(\alpha) = \pi^q\delta_L(\alpha) + \delta_L(\pi)\phi(\alpha).\]
	As $\delta_L(\pi) = 1 - \pi^{q - 1}$ is a unit and
	\[\pi^q\delta_L(\alpha) = \phi(\pi^{q - 1}\alpha) - \pi^{q - 1}\alpha^q = 0\]
	we are done.
\end{proof}
\begin{lem}\label{lem:distinguished-elt-nonzerodivisor}
	If $A$ is a perfect and $\pi$-adically complete $\delta_L$-algebra, and $d$ is distinguished, then $d$ is a nonzerodivisor.
\end{lem}
\begin{proof}
	Suppose that $d\alpha = 0$ and suppose towards a contradiction that $\alpha\neq 0$. Since $A$ is $\pi$-torsionfree by lemma~\ref{lem:perfect-implies-torsionfree} and $\pi$-adically complete, we can further assume that $\pi\nmid \alpha$. Applying $\delta_L$ to $d\alpha = 0$ gives
	\[\alpha^q\delta_L(d) + \delta_L(\alpha)\phi(d) = 0.\]
	Multiplying by $\phi(\alpha)$ and using that $d$ is distinguished then implies $\alpha^q\phi(\alpha) = 0$. Thus $\alpha^{2q}\equiv 0\pmod{\pi}$. But as $\phi$ is a bijection, $\varphi_q$ is injective, so $\pi|\alpha$, a contradiction.
\end{proof}

A key fact about perfect $\delta_L$-algebras is the following.
\begin{prop}{(See \cite[Corollary~2.31]{BS}.)}\label{prop:perfect-d-algs-equiv} The following functors are equivalences of categories.
\begin{center}
	\begin{tikzcd}
		\left\{\begin{array}{c} \text{$\pi$-adically complete} \\ \text{perfect $\delta_L$-algebras}\end{array}\right\} \ar{r}{forget}
		& \left\{\begin{array}{c} \text{$\pi$-adically complete} \\ \text{$\pi$-torsion free $\O_L$-algebras $A$} \\ \text{with $A/\pi$ perfect}\end{array}\right\}\ar{r}{A\mapsto A/\pi}
		& \left\{\text{perfect $\F_q$-algebras}\right\}\ar[bend left]{ll}{W_L(R)\mapsfrom R}
	\end{tikzcd}
\end{center}
\end{prop}
\begin{proof}
	By lemma~\ref{lem:perfect-implies-torsionfree}, the forgetful functor has image in $\pi$-torsion free rings. By the vanishing of the cotangent complex $\mathbb{L}_{R/\F_q}$ for a perfect $\F_q$-algebra $R$ and deformation theory, there is a unique $\pi$-adically complete and $\pi$-torsion-free $\O_L$-algebra $\tilde{R}$ such that $\tilde{R}/\pi\cong R$. Since $R\mapsto \tilde{R}$ is clearly quasi-inverse to $A\mapsto A/\pi$, it suffices to show that $\tilde{R}$ is naturally isomorphic to $\mathrm{forget}(W_L(R))$.

	Since $R\mapsto \tilde{R}$ is a functor, $\tilde{R}$ comes equipped with a canonical lift of $q$-Frobenius and thus by lemma~\ref{lem:W-delta_pi-functor} a canonical map $s_{\tilde{R}}:\tilde{R}\rightarrow W_L(R)$ lifting $\tilde{R}\twoheadrightarrow \tilde{R}/\pi = R$. By \cite[prop.~1.1.18]{SchnBook}, $W_L(R)$ is $\pi$-adically complete, so it suffices to show that $s_{\tilde{R}}$ induces an isomorphism $R\rightarrow W_L(R)/\pi$. But this is clear since an inverse is given by the map $W_L(R)/\pi\rightarrow R/\pi = R$ induced by $W_L(R)\twoheadrightarrow R$. 
\end{proof}

\begin{cor}\label{cor:W_L-perfect-alg}
	If $R$ is a perfect $\F_q$ algebra, then $W_L(R)\cong W(R)\otimes_{W(\F_q)}\O_L$ where $W$ denotes the $p$-typical Witt vectors. In particular $W_L(\F_q) = \O_L$.
\end{cor}

\section{\texorpdfstring{$L$}{L}-typical prisms}\label{sec:L-typical-prisms}

As before, let $L/\Q_p$ be a finite extension with uniformizer $\pi$ and residue field $\F_q$. In this section, we introduce $L$-typical prisms, which are a mild generalization of prisms obtained by replacing $\delta$-rings with $\delta_L$-algebras. In \secsymb\ref{ssec:L-typical-basic-theory} we define $L$-typical prisms and the $L$-typical prismatic site of a formal scheme $X$ over $\Spf\O_L$. 

Prisms as defined can be viewed as ``deperfections'' of perfectoid rings, in the sense that the subcategory of perfect prisms is equivalent to the category of perfectoid rings. Similarly, in \secsymb\ref{ssec:perfect-prisms-perfectoid-algebras} we show that the category of $L$-typical prisms has a subcategory of perfect $L$-typical prisms, which are equivalent to \textit{perfectoid $\O_L$-algebras} (i.e. perfectoid rings with an $\O_L$-algebra structure). We also show that there is a perfection functor for $L$-typical prisms.

In \secsymb\ref{ssec:prism-maps}, we give two constructions which -- given an $L$-typical prism $(A,I)$, a perfectoid $\O_L$-algebra $R$, and a system of $\phi$-compatible maps $(\iota_n:A\rightarrow R)_n$ -- produce a map $(A,I)\rightarrow (\Ainf(R),\ker\theta)$ to the perfect $L$-typical prism corresponding to $R$. These constructions will play a crucial role in \secsymb\ref{sec:LT-phi-G-modules}, where they are used to embed an $L$-typical prism coming from a Lubin-Tate formal $\O_L$-module inside a perfect $L$-typical prism.

\subsection{Basic theory}\label{ssec:L-typical-basic-theory}


\begin{defn}~
	\begin{enumerate}
		\item[(1)] An \textit{$L$-typical prism} is a pair $(A, I)$ where $A$ is a $\delta_L$-algebra and $I\subseteq A$ is an ideal defining a Cartier divisor on $\Spec(A)$ such that $A$ is derived $(\pi,I)$-complete and $\pi\in I + \phi_A(I)$. A morphism $(A,I)\rightarrow (B, J)$ of prisms is a $\delta_L$-algebra morphism $f:A\rightarrow B$ such that $f(I)\subseteq J$.
		\item[(2)] An $L$-typical prism $(A,I)$ is \textit{perfect} if $A$ is a perfect $\delta_L$-algebra. It is \textit{bounded} if $A/I$ has bounded $\pi^\infty$-torsion, i.e. $A/I[\pi^\infty] = A/I[\pi^n]$ for some $n\ge 0$. 
		\item[(3)] If $X$ is a formal scheme over $\Spf\O_L$ then the (absolute) $L$-typical prismatic site $X_{\Prism_L}$ has
			\begin{itemize}
			 	\item objects: bounded $L$-typical prisms $(A,I)$ together with a map of formal schemes $\Spf(A/I)\rightarrow X$;
			 	\item morphisms: maps of $L$-typical prisms compatible with the structure map to $X$;
			 	\item covers: morphisms $(A,I)\rightarrow (B,J)$ such that $A\rightarrow B$ is $(\pi,I)$-completely faithfully flat.
			 \end{itemize}
			 If $X = \Spf(R)$ then we write $R_{\Prism_L}$ for $X_{\Prism_{L}}$.
	\end{enumerate}
\end{defn}
The same definition was independently given in the concurrent work of Ito \cite{Ito}, where $L$-typical prisms are called $\O_L$-prisms in Ito's terminology.
\begin{rmk}
	For the notions of derived $I$-completeness and $I$-complete faithful flatness, see \cite[\secsymb1.2]{BS}. Note that \cite{Wu} omits the word ``faithfully'' in the definition of a cover.
\end{rmk}

As suggested by the definition of the prismatic site, we will only be interested in bounded prisms. In this case, we need not worry about the word ``derived'' in the definition of a $L$-typical prism.
\begin{lem}
	If $(A,I)$ is a bounded $L$-typical prism, then $A$ is classically $(\pi, I)$ complete.
\end{lem}
\begin{proof}
	This is the same as in \cite[lem.~3.7]{BS}. In more detail, we may suppose that $I = (d)$ for a nonzerodivisor $d$. Then by the derived $(\pi,d)$-completeness of $A$, the fact that $A/d^m$ has bounded $\pi$-torsion for all $m$ (by devissage), and \cite[\href{https://stacks.math.columbia.edu/tag/091X}{Tag 091X}]{stacks-project}, we have
	\begin{align*}
	A&\cong R\lim_mR\lim_n (A\otimes_{\Z[d]}^L\Z[d]/(d^m))\otimes_{\Z[\pi]}^L \Z[\pi]/(\pi^n) \\
	&\cong R\lim_mR\lim_n A/(d^m)\otimes_{\Z[\pi]}^L \Z[\pi]/(\pi^n)\cong \lim_m\lim_n A/(d^m,\pi^n)
	\end{align*}
	as desired.
\end{proof}

If $(A,I)$ is a $L$-typical prism with $I$ \textit{principal}, then lemma~\ref{lem:dist-characterization} shows that the condition $\pi\in (I,\phi(I))$ is equivalent to $I$ being generated by a distinguished element. Under the weaker assumption that $I$ is Zariski-locally principal, the condition $\pi\in (I,\phi(I))$ is equivalent to $I$ being ind-Zariski-locally generated by a distinguished element. (The 'ind-' is necessary because after after passing to a Zariski open, we may no longer have $(\pi,I)\subseteq \mathrm{Rad}(A)$, which necessitates passing to a further localization along $(\pi, I)$; see \cite[footnote 8]{BS} for more details.)

\begin{lem}\label{lem:ind-zariski-locally-distinguished}
	Let $A$ be a $\delta_L$-algebra and $I\subseteq A$ a Zariski-locally principal ideal such that $(\pi,I)\subseteq \mathrm{Rad}(A)$. Then the following are equivalent:
	\begin{enumerate}
		\item[(i)] $\pi\in (I^q,\phi(I))$.
		\item[(ii)] $\pi \in (I,\Phi(I))$.
		\item[(iii)] There is a faithfully flat map of $\delta_L$-algebras $A\rightarrow A'$ with $A'$ an ind-(Zariski localization) of $A$ such that $IA'$ is generated by a distinguished element $d$ and $(\pi, d)\in \mathrm{Rad}(A')$.
	\end{enumerate}
\end{lem}
\begin{proof}
	We follow \cite[lem.~3.1]{BS}. Clearly (i)$\implies$(ii). For (ii)$\implies$(iii), since $I$ is locally principal we can select $f_1,\dots, f_n\in A$ generating the unit ideal in $A$ such that each $IA[1/f_i]$ is principal. Take $A' = \left(\prod A[1/f_i]\right)_{(\pi,I)}$, where the subscript denotes Zariski localization along $V((\pi, I))$. Then $A'$ has a unique $\delta_L$-structure by the $L$-typical analogue of \cite[rmk.~2.16]{BS}, $A\rightarrow A'$ is a faithfully flat of $\delta_L$-algebras, and $I' = IA'$ is principal with $\pi\in (I',\phi(I'))$. By lemma~\ref{lem:dist-characterization}, any generator of $I'$ is distinguished.

	For (iii)$\implies$(i), we would like to check that $\pi = 0$ in $A/(I^q,\phi(I))$. This can be checked after faithfully flat extension to $A'$, in which case it follows from lemma~\ref{lem:dist-characterization}.
\end{proof}

Even though $I$ is only assumed locally principal, $\phi(I)$ is always principal.
\begin{lem}\label{lem:phi(I)-principal}
	The ideal $\phi(I)$ is principal and generated by a distinguished element for any $L$-typical prism $(A,I)$.
\end{lem}
\begin{proof}
	By lemma~\ref{lem:ind-zariski-locally-distinguished}, we can pick $a\in I^q, b\in \phi(I)$ so that $\pi = a + b$. We will show that $b$ generated $\phi(I)$. This can be checked after passing to the ind-Zariski-localization $A'$ of lemma~\ref{lem:ind-zariski-locally-distinguished}. Let $d$ be a distinguished generator of $IA'$ so that $a = \alpha d^q$ and $b = \beta\phi(d)$ in $A'$; it suffices to show that $\beta$ is a unit. Indeed, applying $\delta_L$ to the equation $\pi = \alpha d^q + \beta\phi(d)$ and working mod $(\pi, d)$ gives
	\[\delta_L(\pi)\equiv \beta^q\delta_L(d)^q\pmod{(\pi,d)},\]
	which implies that $\beta$ is a unit in $A'/(\pi,d)$ and hence in $A'$, as desired.
\end{proof}

\subsection{Perfect \texorpdfstring{$L$}{L}-typical prisms and perfectoid \texorpdfstring{$\O_L$}{O\_L}-algebras}\label{ssec:perfect-prisms-perfectoid-algebras}

It is shown in \cite[Theorem~3.10]{BS} that the functor $(A,I)\mapsto A/I$ is (one half of) an equivalence of categories between perfect prisms and perfectoid rings. In fact, this can be taken as the definition of a perfectoid ring, as is done in \cite[IV]{bhattnotes}. We take the same perspective here, initially \textit{defining} perfectoid $\O_L$-algebras as those $\O_L$ algebras which come from perfect $L$-typical prisms. We will later show (theorem~\ref{thm:perfectoid-OL-alg-perfectoid-ring}) that one can equivalently define perfectoid $\O_L$-algebras as perfectoid rings equipped with an $\O_L$-algebra structure.

\begin{defn}
	An $\O_L$-algebra $R$ is a \textit{perfectoid $\O_L$-algebra} if it is isomorphic to $A/I$ for some perfect $L$-typical prism $(A,I)$.
\end{defn}

The functor from pefectoid rings to perfect prisms is $R\mapsto (\Ainf(R), \ker \theta)$. To generalize this functor to the present setting, recall that the \textit{tilt} of a ring $R$ is $R^\flat = \plim_{\varphi_p}R/p.$ If $R$ is an $\O_L$-algebra, then we have an isomorphism of rings
\[R^\flat \cong \plim_{\varphi_q} R/\pi\]
so that $R^\flat$ is in fact a perfect $\F_q$-algebra. If $R$ is moreover $\pi$-adically complete then we have an isomorphism of multiplicative monoids $R^\flat \stackrel{\sim}{\rightarrow} \plim_{x\mapsto x^q} R$; by composing this with projection onto the first factor of the inverse limit, we get a multiplicative map $\sharp: R^\flat\rightarrow R$ explicitly given by
\[x^\sharp = \lim_{n\rightarrow\infty}\widehat{x_n}^{q^n}\quad\quad\text{where }x = (\dots, x_1, x_0)\in \plim_{\varphi_q}R/\pi = R^\flat\]
and where the $\widehat{x_n}\in R$ are arbitrary lifts of the $x_n\in R/\pi$.

\begin{defn}
	If $R$ is a $\pi$-adically complete $\O_L$-algebra, then let $\Ainf(R) = W_L(R^\flat)$ and $\theta:\Ainf(R)\rightarrow R$ be the map given in Witt coordinates by
	\[(x_0, x_1, \dots)\mapsto \sum_{n\ge 0} \left(x_n^{1/q^n}\right)^\sharp \pi^n.\]
	By corollary~\ref{cor:W_L-perfect-alg} we have $W_L(R^\flat)\cong W(R^\flat)\otimes_{W(\F_q)}\O_L$, and $\theta$ is a ring homomorphism coinciding with the base change to $\O_L$ of the usual map $\theta:W(R^\flat)\rightarrow R$ of $p$-adic Hodge theory.
\end{defn}
\begin{rmk}
	If $L/\Q_p$ is unramified and $\pi = p$, then $\Ainf(R)$ and $\theta$ coincide with their usual meanings, but this is not the case when $\O_L/\Q_p$ is ramified.
\end{rmk}

\begin{lem}\label{lem:theta-surj-Ainf-complete}
	Let $R$ be a $\pi$-adically complete $\O_L$-algebra with $\varphi_q:R/\pi\rightarrow R/\pi$ surjective.
	\begin{enumerate}
		\item[(1)] The map $\theta: \Ainf(R)\rightarrow R$ is surjective. 
		\item[(2)] $\Ainf(R)$ is $(\pi, \ker\theta)$-adically complete.
	\end{enumerate}
\end{lem}
\begin{proof}
	First note that $\Ainf(R)$ is $\pi$-adically complete by \cite[prop.~1.1.18]{SchnBook}. Thus part (1) reduces to showing that $R^\flat\rightarrow R/\pi$ is surjective, which follows from the assumption that $\varphi_q$ is surjective.

	For (2), using again the $\pi$-completeness of $\Ainf(R)$, it suffices to check that $R^\flat$ is complete with respect to the ideal $J = \ker(R^\flat\rightarrow R/\pi)$ which is the mod $\pi$ reduction of $\ker\theta$. Indeed, we have $R^\flat = \plim_{\varphi_q}R/\pi\cong \plim_n R^\flat/J^{q^n}$ via the isomorphisms $R/\pi = R^\flat/J\stackrel{\varphi_q^n}{\rightarrow} R^\flat/J^{q^n}$.
\end{proof}

The following properties make perfect $L$-typical prisms especially well-behaved.
\begin{lem}\label{prop:prism-facts}
	Let $(A,I)$ be a perfect $L$-typical prism.
	\begin{enumerate}
		\item[(1)] $I$ is principal and generated by a distinguished element.
		\item[(2)] $(A,I)$ is bounded.
		\item[(3)] $A/I$ is $\pi$-adically complete. 
	\end{enumerate}
\end{lem}
\begin{proof}
	(1) follows from lemma~\ref{lem:phi(I)-principal}. Let $d\in A$ be the distinguished generator of $I$.

	For (2), we will in fact show that $A/d[\pi^2] = A/d[\pi]$. Suppose that $\alpha\in A/d[\pi^2]$ so that there is some $\beta\in A$ with $\pi^2\alpha = \beta d$. Applying $\delta_L$ and working mod $\pi$, we get that
	\[d^q\delta_L(\beta) + \beta^q\delta_L(d)\equiv 0\pmod{\pi}.\]
	Multiplying by $\beta^q$ and and using that $\delta_L(d)\in A^\times$ then gives that $\pi|\beta^{2q}$. This implies $\pi|\phi^2(\beta)$, so that $\pi|\beta$ since $\phi$ is an $\O_L$-linear isomorphism. Thus we have
	\[\pi^2\alpha = \pi\beta'd\quad\text{for some }\beta'\in A\] 
	so that $d|\pi\alpha$ by lemma~\ref{lem:perfect-implies-torsionfree}.

	(3) follows from \cite[\href{https://stacks.math.columbia.edu/tag/091X}{Tag 091X}]{stacks-project} since $A/d$ is derived $\pi$-complete with bounded $\pi$-power torsion.

\end{proof}

\begin{prop}\label{prop:perfect-prism-perfectoid-equiv}
	We have an equivalence of categories
	\begin{center}
	\begin{tikzcd}
		\left\{\begin{array}{c} \text{perfect} \\ \text{$L$-typical prisms}\end{array}\right\} \ar[bend left]{r}{(A,I)\mapsto A/I}
		& \left\{\begin{array}{c} \text{perfectoid} \\ \text{$\O_L$-algebras} \end{array}\right\}\ar[bend left]{l}{(\Ainf(R),\ker\theta)\mapsfrom R}
	\end{tikzcd}
\end{center}
\end{prop}
\begin{proof}
	Let $R = A/I$ be a perfectoid $\O_L$-algebra coming from a perfect $L$-typical prism $(A,I)$. Since $R$ is $\pi$-adically complete by proposition~\ref{prop:prism-facts} and $\varphi_q:R/\pi\rightarrow R/\pi$ is surjective (as it's the mod $(\pi,I)$ reduction of $\phi:A\rightarrow A$), lemma~\ref{lem:theta-surj-Ainf-complete}(1) implies that $\theta:\Ainf(R)\rightarrow R$ is surjective. Thus to prove the proposition, it suffices to show that $\Ainf(R)$ identifies with $A$ in such a way that
	\begin{center}
	\begin{tikzcd}
		A\ar[-]{rr}{\sim}\ar[twoheadrightarrow]{dr} & & \Ainf(R)\ar[twoheadrightarrow]{dl}{\theta} \\
		&A/I = R
	\end{tikzcd}
	\end{center}
	commutes (thereby identifying $I$ with $\ker \theta$). Since $\Ainf(R)$ and $A$ are $\pi$-adically complete perfect $\delta_L$-algebras, by proposition~\ref{prop:perfect-d-algs-equiv} it suffices to show that $A/\pi$ identifies with $R^\flat$ compatibly with the maps to $A/(\pi, I) = R/\pi$. Indeed, we have a commutative diagram
	\begin{center}
		\begin{tikzcd}
			A/\pi\ar[-]{r}{\sim}\ar[twoheadrightarrow]{dr} & \plim_{\varphi_q} A/(\pi, I)\ar[twoheadrightarrow]{d} & R^\flat\ar[-,swap]{l}{\sim}\ar[twoheadrightarrow]{dl} \\
			& A/(\pi, I)
		\end{tikzcd}
	\end{center}
	via the $I$-adic completeness of $A/\pi$.
\end{proof}

\begin{lem}\label{lem:pi-faithfully-flat-pi-I-faithfully-flat}
	A map $R\rightarrow S$ of perfectoid $\O_L$-algebras is $\pi$-completely (faithfully) flat if and only if the corresponding map$\Ainf(R)\rightarrow \Ainf(S)$ is $(\pi,\ker\theta)$-completely (faithfully) flat.
\end{lem}
\begin{proof}
	It is easy to show that $\Ainf(S)\otimes^L_{\Ainf(R)}\Ainf(R)/\ker\theta_{\Ainf(R)}\cong S\otimes_{R}^LR$ using either the $L$-typical analogue of the rigidity result \cite[lemma~3.5]{BS} or the fact that a distinguished element can only factor as a unit times another distinguished element \cite[lemma~2.24]{BS}. Thus $R\rightarrow S$ being $\pi$-competely (faithfully) flat and $\Ainf(R)\rightarrow \Ainf(S)$ being $(\pi,\ker\theta)$-completely (faithfully) flat are both equivalent to $R/\pi\rightarrow S\otimes_R^L R/\pi$ being (faithfully) flat. 
\end{proof}

Given a $L$-typical prism $(A,I)$ we can form its perfection.
\begin{defn}~\label{defn:prism-perfection}
	If $(A,I)$ is a $L$-typical prism, then we write 
	\[A_{\perf} = (\varinjlim_\phi A)^\wedge_{(\pi,I)}\]
	for the (classical) $(\pi,I)$-completion of the naive perfection $\varinjlim_\phi A$. We call $(A_{\perf},IA_{\perf})$ the \textit{perfection} of $(A,I)$.
\end{defn}
By remarks \ref{rmk:lim-colim-closed}(4)-(5), $A_{\perf}$ is a perfect $\delta_L$-algebra. We now show that $(A_{\perf}, IA_{\perf})$ is the initial $L$-typical prism over $(A,I)$. 
\begin{prop}(cf. \cite[Lemma~3.9]{BS})\label{prop:prism-perfection}
	Let $(A,I)$ be a $L$-typical prism.
	\begin{enumerate}
		\item[(1)] The derived $(\pi,I)$-adic completion of $\varinjlim_\phi A$ coincides with the classical $(\pi,I)$-adic completion (and thus with $A_{\perf}$).
		\item[(2)] The map $(A,I)\rightarrow (A_{\perf},IA_{\perf})$ is initial among maps from $(A,I)$ to a perfect $L$-typical prism. 
	\end{enumerate}
\end{prop}
\begin{proof}
	(1) clearly implies (2). 
	To show (1) first note that by construction $\varinjlim_\phi A$ is a perfect $\delta_L$-algebra. Thus by lemma~\ref{lem:perfect-implies-torsionfree} $A$ is $\pi$-torsionfree, so that the derived and classical $\pi$-adic completions agree. As $A\rightarrow (\varinjlim_\phi A)^\wedge_{(\pi)}$ factors through $\phi:A\rightarrow A$, lemma~\ref{lem:phi(I)-principal} implies that $I(\varinjlim_\phi A)^\wedge_{(\pi)}$ is principal and generated by a distinguished element $d$. By lemma~\ref{lem:dist-characterization} 
	it thus suffices to show that $d$ is a nonzerodivisor.

	For this, suppose that $fd = 0$ for some $0\neq f\in (\varinjlim_\phi A)^\wedge_{(\pi)}$; since this ring is $\pi$-torsionfree and classically $\pi$-adically complete (and thus $\pi$-adically separated), we can suppose that $\pi\nmid f$ by dividing out powers of $\pi$. Applying $\delta_L$ and working mod $\pi$ we get
	\[f^q\delta_L(d) + d^q\delta_L(f)\equiv 0\pmod{\pi}.\]
	Multiplying by $f^q$ and using that $\delta_L(d)$ is a unit then shows that $\pi|f^{2q}$. Thus $\pi|\phi^2(f)$, which implies that $\pi|f$ since $\phi$ is a $\O_L$-linear isomorphism. But this is a contradiction, so $d$ must be a nonzerodivisor. 
\end{proof}

Finally, we will show that perfectoid $\O_L$-algebras coincide with perfectoid rings (in the sense of \cite[definition~3.5]{BMS}) equipped with an $\O_L$-algebra structure. We begin by establishing a more intrinsic criterion for being a perfectoid $\O_L$-algebra; indeed the following proposition is the $L$-typical version of \cite[proposition~IV.2.10]{bhattnotes}, in which $p$ is replaced by $\pi$ and the Frobenius is replaced by the $q$-Frobenius.

\begin{prop}~\label{prop:perfectoid-charactierzation}
	Let $R$ be an $\O_L$-algebra. Then $R$ is a perfectoid $\O_L$-algebra if and only if
	\begin{enumerate}
		\item[(1)] $R$ is $\pi$-adically complete,
		\item[(2)] there exists some $\varpi\in R$ such that $\varpi^q = \pi u$ for some $u\in R^\times$,
		\item[(3)] $\varphi_q:R/\pi\rightarrow R/\pi$ is surjective, and
		\item[(4)] the kernel of $\theta:\Ainf(R)\rightarrow R$ is principal. 
	\end{enumerate}
	If $R$ is assumed $\pi$-torsionfree, then the above remains true with (4) replaced by
	\begin{enumerate}
		\item[(4')] if $x\in R[1/\pi]$ with $x^q\in R$, then $x\in R$. 
	\end{enumerate}
\end{prop}
\begin{rmk}~\label{rmk:q-power-roots}
	Note that once a $q$th root of $\pi u$ as in (2) exists, we get a full $q$-power-compatible system of roots $(\varpi^{1/q^n})$ by letting $\varpi^\flat\in R^\flat$ be any lift of $\varpi$ along the map $R^\flat\rightarrow R/\pi$ (which is surjective by assumption (3)) and then taking $(\varpi^{1/q^n})$ to be the image of $\varpi^\flat$ under the bijection $R^\flat\cong \plim_{x\mapsto x^q} R$ (which exists by assumption (1)).
\end{rmk}
\begin{proof}[Proof of proposition~\ref{prop:perfectoid-charactierzation}.]
	Suppose that $R = A/I$ is a perfectoid $\O_L$-algebra coming from a perfect $L$-typical prism $(A,I)$; using proposition~\ref{prop:perfect-prism-perfectoid-equiv} we can identify $(A,I)\cong (\Ainf(R),\ker\theta)$. (1) and (4) follow from proposition~\ref{prop:prism-facts}, and (3) follows from the surjectivity of $\phi:A\rightarrow A$. For (2), let $d\in \Ainf(R)$ be a distinguished generator of $\ker\theta$ (which exists by proposition~\ref{prop:prism-facts}). Then we can take $\varpi = \theta(\phi^{-1}(d))$ since
	\[\varpi^q = \theta\left(\phi^{-1}(d^q)\right) = \theta\left(d - \pi\phi^{-1}(\delta_L(d))\right) = -\pi\theta(\phi^{-1}(\delta_L(d))),\]
	with $u = -\theta(\phi^{-1}(\delta_L(d)))\in R^\times$. 

	\begin{rmk}~\label{rmk:varpi-generates-ker-varphiq}
		Using the diagram
		\begin{center}
			\begin{tikzcd}
				\Ainf(R)/(\pi,d)\ar["\sim", "\theta"']{r}\ar{d}{\phi} & R/\pi\ar{d}{\varphi_q} \\
				\Ainf(R)/(\pi,d)\ar["\sim", "\theta"']{r} & R/\pi
			\end{tikzcd}
		\end{center}
		we note that the mod $\pi$ reduction of the element $\varpi$ constructed above is the kernel of $\varphi_q:R/\pi\rightarrow R/\pi$. Thus the surjective map $\varphi_q:R/\pi\rightarrow R/\pi = R/\varpi^q$ factors through an isomorphism $R/\varpi\stackrel{\sim}{\rightarrow} R/\varpi^q$. This fact will be used later in the proof.
	\end{rmk}

	For the converse, suppose that $R$ is an $\O_L$-algebra satisfying (1)-(4); we want to show that $(\Ainf(R), \ker\theta)$ is a perfect $L$-typical prism. $\Ainf(R) = W_L(R^\flat)$ is a perfect $\delta_L$-algebra by proposition~\ref{prop:perfect-d-algs-equiv} and is $(\pi,\ker\theta)$-adically complete by lemma~\ref{lem:theta-surj-Ainf-complete}. By assumption $\ker\theta = (d)$ for some $d\in \Ainf(R)$. Thus by lemma~\ref{lem:dist-characterization} it suffices to show that $d$ is distinguished.

	Let $\varpi, u\in R$ be as in (2), and let $\omega,v \in \Ainf(R)$ be lifts along $\theta$. Then $\omega^q - \pi v\in \ker\theta$, so we can write
	\[\omega^q - \pi v = \alpha d\]
	for some $\alpha\in \Ainf(R)$. Applying $\delta_L$ to this equation and working mod $(\pi, d)$ (using remark~\ref{rmk:q-derivation}(2) to simplify) gives
	\[-v\delta_L(\pi)\equiv \alpha^q\delta_L(d)\pmod{(\pi, d)}.\]
	As $-v\delta_L(\pi)\in \left(\Ainf(R)/(\pi, d)\right)^\times$, this shows that $\delta_L(d)\in \left(\Ainf(R)/(\pi,d)\right)^\times$ as well, and thus $\delta_L(d)\in \Ainf(R)^\times$ by $(\pi, d)$-completeness.

	Assume now that $R$ is $\pi$-torsion free. Supposing $R$ is a perfectoid $\O_L$-algebra, we prove (4'). Suppose that $x\in R[1/\pi]$ with $x^q\in R$. Let $\varpi\in R$ be the element satisfying (2) constructed earlier in this proof; by remark~\ref{rmk:varpi-generates-ker-varphiq} we have that the $q$-power map $R/\varpi\rightarrow R/\varpi^q$ is bijective. Let $n\ge 0$ be minimal such that $\varpi^n x\in R$ (such an $n$ exists since $\varpi^q|\pi$), and suppose towards a contradiction that $n\ge 1$. Then 
	\[(\varpi^nx)^q = \varpi^{nq}x^q\in \varpi^{nq}R\subseteq \varpi^q R,\]
	which implies that $\varpi^nx\in \varpi R$. As $R$ is $\varpi$-torsionfree, this implies that $\varpi^{n-1}x\in R$, giving the contradiction.

	Finally, suppose that $R$ is a $\pi$-torsionfree $\O_L$-algebra satisfying (1) - (3) and (4'); we will prove (4). Let $\varpi\in R$ be as in (2), and let $(\varpi^{1/q^n})$ be a system of $q$-power roots of $\varpi$, which exists by remark~\ref{rmk:q-power-roots} (which uses only (1)-(3)). We have that $\varpi^{1/q^n}$ mod $\pi$ generates $\ker(\varphi_q^n:R/\pi\rightarrow R/\pi)$: if $x\in R$ with $\pi|x^{q^n}$ then the $q^n$-th power of $x/\varpi^{1/q^n}\in R[1/\pi]$ is in $R$, so that $x/\varpi^{1/q^n}\in R$ as well by assumption. It follows that the element
	\[(\dots, \overline{\varpi^{1/q^2}}, \overline{\varpi^{1/q}}, \overline{\varpi}, 0)\in R^\flat\]
	formed from the mod $\pi$ reductions of the $\varpi^{1/q^n}$ generates $\ker(R^\flat\rightarrow R/\pi) = \plim \ker(\varphi_q^n)$. But since $R$ is $\pi$-torsionfree and $\ker \theta$ is $\pi$-adically complete with mod $\pi$ reduction $\ker(R^\flat\rightarrow R/\pi)$, this implies that $\ker\theta$ is principal as well.
\end{proof}

\begin{thm}\label{thm:perfectoid-OL-alg-perfectoid-ring}
	Let $R$ be a ring. Then $R$ is a perfectoid $\O_L$-algebra if and only if $R$ is a perfectoid ring and an $\O_L$-algebra.
\end{thm}
\begin{proof}
	First, suppose that $R$ is a perfectoid ring and an $\O_L$-algebra. To show that $R$ is a perfectoid $\O_L$-algebra, it suffices to show that $(\Ainf(R),\ker\theta)$ is an $L$-typical prism (it is automatically perfect by proposition~\ref{prop:perfect-d-algs-equiv}). This is done in \cite[lemma~2.4.3]{Ito}; we briefly sketch the argument here. First, one shows that $\ker(\theta:\Ainf(R^\flat)\rightarrow R)$ is generated by an element of the form $\xi = \pi - [\varpi^\flat]b$, where $\varpi\in R$ is such that $R$ is $\varpi$-adically complete and $\varpi^p|p$, the element $\varpi^\flat\in R^\flat$ satisfies $(\varpi^\flat)^\sharp = \varpi$, and $b\in \Ainf(R)$. Since any generator of $\ker(W(R^\flat)\rightarrow R)$ is a nonzerodivisor and $W(\F_q)\rightarrow \O_L$ is flat, any generator of $\ker\theta$ ia a nonzerodivisor. It is easy to show that $\Ainf(R)$ is $(\pi,\xi)$-adically complete. And $\xi$ is distinguished as $\delta_L(\xi)\equiv 1 - \pi^{q - 1}\pmod{(\pi,\xi)}$ is a unit in $R$.

	Conversely, suppose that $R$ is a perfectoid $\O_L$-algebra; we want to show that $R$ is a perfectoid ring. By lemma~\ref{lem:perfectoid-structure} below, we have that $R$ can be written as a fiber product $\overline{R}\times_{\overline{S}}S$ where $\overline{R}\rightarrow \overline{S}$ is a surjection of perfect $\F_q$-algebras, and $S$ is a $\pi$-torsionfree perfectoid $\O_L$-algebra.	Once we show that $\overline{R}$, $\overline{S}$, and $S$ are perfectoid rings, we may conclude that $R$ is a perfectoid ring as well by \cite[proposition~2.1.4]{cesnavicius2023purity}. Thus we may assume that $R$ is a perfect $\F_q$-algebra or $\pi$-torsionfree. In the former case, the result is clear since perfect $\F_p$-algebras are perfectoid rings.

	Suppose now that $R$ is $\pi$-torsionfree. By \cite[proposition~IV.2.10]{bhattnotes} it suffices to show that $R$ satisfies the ``$p$-analogues'' of properties (1)-(3),(4') in proposition~\ref{prop:perfectoid-charactierzation}:
	\begin{enumerate}
		\item[(1\textsubscript{p})] $R$ is $p$-adically complete,
		\item[(2\textsubscript{p})] there exists some $\varpi'\in R$ such that $(\varpi')^p = p u'$ for some $u'\in R^\times$,
		\item[(3\textsubscript{p})] $\varphi:R/p\rightarrow R/p$ is surjective, and
		\item[(4'\textsubscript{p})] if $x\in R[1/p]$ with $x^p\in R$, then $x\in R$.  
	\end{enumerate}
	(1\textsubscript{p}) and (4'\textsubscript{p}) follow immediately from (1) and (4'), respectively. Taking $\varpi,u\in R$ as in (2) and letting $e$ be the ramification index of $L/\Q_p$, we conclude (2\textsubscript{p}) by taking $\varpi' = \varpi^{qe/p}$ and $u' = u^e$. To show (3\textsubscript{p}), it suffices to show the $q$-Frobenius $\varphi_q:R/p\rightarrow R/p$ is surjective. So let $\alpha\in R$; we will successively approximate $\alpha^{1/q}$ modulo higher powers of $\pi$. Indeed, by (3) there is $\beta_1\in R$ so that $\beta_1^q\equiv \alpha\pmod{\pi}$. Thus $\frac{\alpha - \beta_1^q}{u\pi}\in R$, where $u\in R^\times$ is as in (2). Again by (3), there is $\gamma_1\in R$ so that $\gamma_1^q\equiv \frac{\alpha - \beta_1^q}{u\pi}\pmod{\pi}$. Then let $\beta_2 = \beta_1 + \gamma_1\varpi$ with $\varpi$ as in (2). We see that $\beta_2^q\equiv \beta_1^q + \gamma_1^qu\pi = \alpha\pmod{\pi^2}$ (so long as the ramification index $e\ge 2$; if $e = 1$ then $\pi = p$, so we were already done). Repeating this process, we get for $1\le n\le e$ elements $\beta_n\in R$ satisfying $\beta_n^q\equiv \alpha\pmod{\pi^n}$. We then get that $\beta_e^q\equiv \alpha\pmod{p}$ as desired.
\end{proof}

\begin{lem}\label{lem:perfectoid-structure}
	(cf. \cite[proposition~IV.3.2]{bhattnotes}.) Let $R$ be a perfectoid $\O_L$-algebra. Then the $\O_L$-algebras
	\[\overline{R} = R/\sqrt{\pi R},\quad S = R/R[\sqrt{\pi R}],\quad\text{and}\quad \overline{S} = S/\sqrt{\pi S}\]
	are perfectoid $\O_L$-algebras and the commutative square
	\begin{center}
		\begin{tikzcd}
			R\ar{r}\ar{d} & S\ar{d} \\
			\overline{R}\ar{r} & \overline{S}
		\end{tikzcd}
	\end{center}
	is Cartesian.
\end{lem}
\begin{proof}
	This is proved the same way as in \cite{bhattnotes}, so we only sketch the argument here. As in the proof of proposition~\ref{prop:perfectoid-charactierzation}, we can write a distinguished generator of $\ker(\theta:\Ainf(R)\rightarrow R)$ as $d =[a_0] - \pi u$ for $a_0\in R^\flat$ such that $R^\flat$ is $a_0$-adically complete and $a_0^\sharp = \pi$, and $u\in \Ainf(R)^\times$ (concretely $a_0 = (\varpi^\flat)^q$ for $\varpi^\flat$ as in remark~\ref{rmk:q-power-roots}). Let $I = (a_0^{1/q^\infty})\subseteq R^\flat$ and $J = R^\flat[I]$. Then the commutative square
	\begin{center}
		\begin{tikzcd}
			W_L(R^\flat)\ar{r}\ar{d} & W_L(R^\flat/J)\ar{d} \\
			W_L(R^\flat/I)\ar{r} & W_L(R^\flat/(I + J)) 
		\end{tikzcd}
	\end{center}
	is a homotopy fibre square by the general result \cite[lemma~IV.3.1]{bhattnotes} regarding perfect $\F_p$ algebras and devissage. Since $d$ is a nonzerodivisor in all of these perfect $\delta_L$-algebras by lemma~\ref{lem:distinguished-elt-nonzerodivisor}, the square remains a homotopy fibre square upon application of $-\otimes_{W_L(R^\flat)}^LR$.
	\begin{center}
		\begin{tikzcd}
			W_L(R^\flat)/(d)\ar{r}\ar{d} & W_L(R^\flat/J)/(d)\ar{d} \\
			W_L(R^\flat/I)/(d)\ar{r} & W_L(R^\flat/(I + J))/(d) 
		\end{tikzcd}
	\end{center}
	It is easy to see that the rings in this square are all perfectoid $\O_L$-algebras. Finally, one shows that the this cartesian square identifies with the one in the statement of the lemma; see \cite[IV]{bhattnotes} for the details.
\end{proof}

\subsection{Constructing maps of \texorpdfstring{$\pi$}{pi}-typical prisms}\label{ssec:prism-maps}

In \secsymb\ref{sec:LT-phi-G-modules} we will construct inclusions of $L$-typical prisms coming from Lubin-Tate formal groups into perfect $L$-typical prisms. The construction used can be understood in at least three different ways, one of which is specific to the scenario in \secsymb\ref{sec:LT-phi-G-modules}, and two of which are general constructions for producing maps between $L$-typical prisms. Here we explain the two general constructions.

\begin{const}\label{const:algebraic-construction}
	Let $R$ be an $\O_L$-algebra, and let $A$ be a $\delta_L$-algebra with a sequence of $\phi$-compatible $\O_L$-algebra maps $\iota_n:A\rightarrow R$ for $n\ge 0$, i.e. a sequence of maps $\iota_n$ making the diagram
	\begin{center}
		\begin{tikzcd}[sep=large]
			A\ar{r}{\phi}\ar{d}{\iota_0}&A\ar{r}{\phi}\ar{dl}{\iota_1} &A\ar{r}{\phi}\ar[bend left=15]{dll}{\iota_2} &\cdots \\
			R
		\end{tikzcd}
	\end{center}
	commute. We will construct from this data a map $\iota: A\rightarrow W_L(R^\flat)$ of $\delta_L$-algebras.

	Indeed, using that $\varphi_q$ commutes with maps of $\F_q$-algebras, we can form a map
	\begin{align*}
		\overline{\iota}:A/&\pi\longrightarrow \plim_{\varphi_q} R/\pi = R^\flat \\
		&a\longmapsto (\overline{\iota_n}(a))_n
	\end{align*}
	in characteristic $p$, where $\overline{\iota_n}:A/\pi\rightarrow R/\pi$ denotes the mod $\pi$ reduction. Then applying the universal property of $W_L$ (lemma~\ref{lem:W-delta_pi-functor}) to the $\O_L$-algebra map $A\twoheadrightarrow A/\pi\stackrel{\overline{\iota}}{\rightarrow} R^\flat$, we get a $\delta_L$-algebra map $\iota: A\rightarrow W_L(R^\flat)$. Note that this construction is purely $\delta_L$-algebraic; it uses nothing from the theory of prisms.
\end{const}
\begin{rmk}\label{rmk:right-shift}
	Intuitively, we think of the system $(\iota_n)_n$ as giving a way to extract $\phi$-power roots in the perfect $\delta_L$-alebra $W_L(R^\flat)$. More precisely, it's easy to show that $\phi^{-m}(\iota(a))$ coincides with $\iota^{\rightarrow m}(a)$, where $\iota^{\rightarrow m}$ denotes the map produced by applying the above construction to the right-shifted system $(\iota_{n + m})_n$.
\end{rmk}

If $R$ is $\pi$-adically complete then we additionally have a map $\theta:W_L(R^\flat)\rightarrow R$. The following lemma computes the composite $A\stackrel{\iota}{\rightarrow} W_L(R^\flat)\stackrel{\theta}{\rightarrow}R$ (possibly with a $\phi$-twist).

\begin{prop}\label{prop:delta-alg-map-computation}
	Fix all notation as above, with $R$ being a $\pi$-adically complete $\O_L$-algebra. Then for any $n\ge 0$, we have
	\[\theta\circ \phi^{-n}\circ \iota = \iota_n.\]
\end{prop}
\begin{proof}
	By remark~\ref{rmk:right-shift}, it suffices to prove this for $n = 0$; thus we will show that $\theta\circ \iota = \iota_0$. The proof is by direct computation. Fixing some $a\in A$, it suffices to show that $\theta(\iota(a))\equiv \iota_0(a)\pmod{\pi^{k + 1}}$ for all $k\ge 0$.

	We can factor the map $\iota$ as
	\[A\stackrel{s_A}{\longrightarrow} W_L(A)\stackrel{W_L(\overline{\iota})}{\longrightarrow} W_L(R^\flat)\]
	where $s_A$ is the section of proposition~\ref{prop:W_pi-sections}(1). Writing $(s_0,s_1,\dots,) = s_A(a)\in W_L(A)$, we find
	\begin{align}
		\theta(\iota(a)) &= \sum_{n = 0}^\infty\left(\overline{\iota}(\overline{s_n})^{1/q^n}\right)^\# \pi^n = \sum_{n = 0}^\infty \lim_{m\rightarrow\infty}\left(\overline{\iota}(\overline{s_n})_{m}\right)^{\wedge,q^{m - n}}\pi^n \notag \\
		&\equiv \lim_{m\rightarrow\infty}\sum_{n = 0}^k\left(\overline{\iota}(\overline{s_n})_{k + m}\right)^{\wedge,q^{k + m - n}}\pi^n \pmod{\pi^{k + 1}} \notag \\
		&= \lim_{m\rightarrow\infty}w_k\left(\left(\overline{\iota}(\overline{s_0})_{k + m}\right)^{\wedge,q^m}, \dots, \left(\overline{\iota}(\overline{s_k})_{k + m}\right)^{\wedge,q^m}\right). \label{eqn:theta-factor-comp}
	\end{align}
	Here we've written $\overline{s_n}$ for the mod $\pi$ reduction of $s_n\in A$, and for $r = (\dots, r_1, r_0)\in R^\flat$, we've written $(r_n)^\wedge$ for an arbitrary lift of $r_n\in R/\pi$ to $R$. Since
	\[\left(\overline{\iota}(\overline{s_n})_{k + m}\right)^{q^m} = \overline{\iota}(\overline{s_n})_k\equiv \iota_k(s_n)\pmod{\pi},\]
	lemma~\ref{lem:w_n-congruence} below allows us to continue (\ref{eqn:theta-factor-comp}):
	\begin{align*}
		\theta(\iota(a)) &\equiv w_k(\iota_k(s_0), \dots, \iota_k(s_k))\pmod{\pi^{k + 1}} \\
		&= \iota_k(w_k(s_0,\dots, s_k)) = \iota_k(\phi^k(a)) = \iota_0(a) 
	\end{align*}
	which is what we wanted. Here we've used that $\iota_k$ is an $\O_L$-algebra map (and thus commutes with $w_k$) and the fact that $w_k\circ s_A = \phi^k$ from the second part of proposition~\ref{prop:W_pi-sections}(1).
\end{proof}

\begin{lem}\label{lem:w_n-congruence}
	If $R$ is an $\O_L$-algebra, $a_0,\dots, a_k,b_0,\dots, b_k\in R$, and $a_n\equiv b_n\pmod{\pi^s}$ for $n = 0,\dots, k$, then
	\[w_k(a_0,\dots, a_k)\equiv w_k(b_0,\dots, b_k)\pmod{\pi^{s + k}}.\]
\end{lem}
\begin{proof}
	See \cite[lemma~1.1.2]{SchnBook}.
\end{proof}

Proposition~\ref{prop:delta-alg-map-computation} suggests viewing the map $\iota:A\rightarrow W_L(R^\flat)$ as a map of prisms when doing so makes sense, i.e. when $R$ is a perfectoid $\O_L$-algebra and $(A,I)$ is a prism with $I\subseteq \ker\iota_0$. This is the perspective taken in the following construction.

\begin{const}\label{const:prismatic-construction}
	Let $(A,I)$ is a $L$-typical prism, let $R$ be a perfectoid $\O_L$-algebra, and suppose given a map $\iota_0:A/I\rightarrow R$. If $(A,I)$ were assumed perfect, then proposition~\ref{prop:perfect-prism-perfectoid-equiv} would allow us to lift $\iota_0$ into a map $\iota:(A,I)\rightarrow (\Ainf(R),\ker\theta)$, but we do not make this assumption. Instead, we further assume given a collection of $\phi$-compatible $\O_L$-algebra maps $\iota_n:A\rightarrow R$ as above with $I\subseteq \ker\iota_0$; this will allow us to construct such a lift $\iota$.

	Using the $\iota_n$, we can factor $\iota_0:A/I\rightarrow R$ through 
	\[\left(\varinjlim_\phi A\right)/I\stackrel{(\iota_n)_n}{\longrightarrow} R,\] 
	and then, after $\pi$-adically completing, through the map of perfectoid $\O_L$-algebras 
	\[A_{\perf}/IA_{\perf} = (\varinjlim_\phi A)^\wedge_{(\pi)}/I\rightarrow R.\] 
	Applying proposition~\ref{prop:perfect-prism-perfectoid-equiv} to the map $A_{\perf}/IA_{\perf}\rightarrow R$ gives a map of prisms fitting into the following diagram.
	\begin{center}
		\begin{tikzcd}[sep=large]
			A\ar{r}\ar[twoheadrightarrow]{d} & A_{\perf}\ar{r}{\text{using prop. }\ref{prop:perfect-prism-perfectoid-equiv}}\ar[twoheadrightarrow]{d} &\Ainf(R)\ar[twoheadrightarrow]{d} \\ 
			A/I\ar{r} & A_{\perf}/IA_{\perf}\ar{r}{\text{using }(\iota_n)_n} & R
		\end{tikzcd}
	\end{center}
	We take $\iota$ to be the composite along the top row, which a map of $L$-typical prisms $(A,I)\rightarrow (\Ainf(R),\ker\theta)$ by construction.
\end{const}
\begin{prop}\label{prop:map-to-perfect-prism}
	Let $X$ be an adic space over $\Spf\O_L$, let $(A,I)\in X_{\Prism_L}$, and let $R$ be a perfectoid $\O_L$-algebra with a structure map $\Spf R\rightarrow X$ over $\Spf \O_L$. Suppose we have a $\phi$-compatible direct system of $\O_L$-algebra maps $\iota_n:A\rightarrow R$ such that $I\subseteq \ker\iota_0$ and the map $\iota_0: A/I\rightarrow R$ is an $X$-morphism. Then there is morphism
	\[\iota:(A,I)\longrightarrow (W_L(R^\flat),\ker\theta)\]
	in $X_{\Prism_L}$ reducing to $\iota_0:A/I\rightarrow R$. Moreover, the map $A\rightarrow \Ainf(R) = W_L(R^\flat)$ of $\delta_L$-algebras obtained this way coincides with that of construction~\ref{const:algebraic-construction}.
\end{prop}
\begin{proof}
	The map $\iota$ of the proposition is given by construction~\ref{const:prismatic-construction}; it is immediate that the morphism constructed this way respects the structure maps to $X$. To show that this coincides with construction~\ref{const:algebraic-construction}, it suffices to show that the maps $\iota^\delta,\iota^\Prism: A_{\perf}\rightarrow W_L(R^\flat)$ induced by constructions~\ref{const:algebraic-construction} and \ref{const:prismatic-construction}, respectively, coincide. By the definition of the $\Ainf$ functor, $\iota^\Prism$ is $W_L$ of the tilt of $A_{\perf}/I\rightarrow R$, so by proposition~\ref{prop:perfect-d-algs-equiv} it suffices to show that taking $\lambda^\delta$ mod $\pi$ gives
	\[\left(A_{\perf}/I\right)^\flat\rightarrow R^\flat.\]
	This is easy to check using the identification $(A_{\perf}/I)^\flat\cong A_{\perf}/\pi$ from the proof of proposition~\ref{prop:perfect-prism-perfectoid-equiv}.
\end{proof}

\begin{example}\label{example:elliptic}
	Let $R$ be a $p$-adically complete ring, and let $E$ be an ordinary elliptic curve over $R$, and let $E[p^\infty] = \varinjlim E[p^n]$ denote the $p$-divisible group of $E$. Then by the theory of the canonical subgroup, there are lifts $F:E\rightarrow E^{(p)}$ of the relative Frobenius $E/p\rightarrow (E/p)^{(p)}$ and $V:E^{(p)}\rightarrow E$ of the Verschiebung with $VF = [p]$. It follows that we have maps
	\[\dots \longrightarrow \ker F^3\stackrel{[p]}{\longrightarrow} \ker F^2\stackrel{[p]}{\longrightarrow}\ker F.\]
	Note that $A = \plim \O_{\ker F^n}$ (inverse limit taken with respect to the inclusion maps $\ker F^n\hookrightarrow \ker F^{n + 1}$) has a lift of Frobenius given by $\phi = [p]^*$. 

	For $n\ge 1$, suppose $R_n$ are \'etale $R$-algebras with sections $e_n:\Spf R_n\rightarrow \ker F^n$. Then, setting $R_\infty = \left(\varinjlim R_n\right)^\wedge_{(p)}$, we have that the maps
	\[\iota_n:A\twoheadrightarrow \ker F^n\stackrel{e_n^*}{\longrightarrow} R_n\hookrightarrow R_\infty\]
	form a $\phi$-compatible system. Thus construction~\ref{const:algebraic-construction} gives a map $A\rightarrow W(R_\infty^\flat)$. 
\end{example}

\section{Lubin-Tate \texorpdfstring{$(\varphi_q,\Gamma)$}{(phi, Gamma)}-modules}\label{sec:LT-phi-G-modules}

In this section, we introduce the key objects involved in Kisin-Ren's theory of $(\varphi_q, \Gamma)$-modules. Pleasantly, much of this theory can be succintly stated in the prismatic language developed in section~\ref{sec:L-typical-prisms}.

As before, let $L/\Q_p$ be a finite extension with uniformizer $\pi$, and let $q = |\O_L/\pi|$. Let $\G$ be a Lubin-Tate formal $\O_L$-module. 

By a \textit{$p$-adic field} $K/L$ we will mean an algebraic extension such that $\O_K$ is a discrete valuation ring with perfect residue field; equivalently this means that $K$ has a perfect residue field $k$ and $K/W_L(k)[1/p]$ is finite. If $K/L$ is a $p$-adic field and $n\ge 0$ we write
\[K_n = K(\G[\pi^n])\]
for the extension given by adjoining the $\pi^n$-torsion points of $\G$. We also write $K_\infty$ for the $p$-adic completion of $\bigcup K_n$ and $\Gamma_K = \Gal(K_\infty/K)$. The action of the absolute Galois group $G_K$ on the free $\O_L$-rank one Tate module $T\G$ factors through $\Gamma_K$ and gives an injective character $\chi_\G:\Gamma_K\rightarrow \O_L^\times$. If $K = L$ then by local class field theory $\chi_\G$ is an isomorphism $\Gamma_L\stackrel{\sim}{\rightarrow} \O_L^\times$.

Throughout this section, we fix once and for all
\begin{itemize}
	\item a coordinate $T$ on $\G$, so that the action of $\O_L$ on $\G\cong \Spf(\O_L\ps{T})$ is given by power series $[a](T)\in \O_L\ps{T}$ for $a\in \O_L$;
	\item a basis $e = (e_n)_{n\ge 0}$ of the free $\O_L$-module $T\G$, viewed as a sequence of $e_n\in \O_{\overline{K}}$ such that $[\pi](e_n) = e_{n - 1}$, $e_0 = 0$, and $e_1\neq 0$.
\end{itemize}
Note that $[\pi](T)\equiv T^q\pmod{\pi}$ and $K_n = K(e_n)$.

In \secsymb\ref{ssec:S_K}, we'll see that $\O_\G\otimes_{\O_L}W_L(k)\cong W_L(k)\ps{T}$ is a $L$-typical prism in $(W_L(k))_{\Prism_L}$, and that our choice of basis $e\in T\G$ allows us to naturally view this prism inside of the perfect prism $\Ainf(\O_{K_\infty})$ via the constructions in \secsymb\ref{ssec:prism-maps}. Though the map $\O_K\ps{T}\rightarrow \Ainf(\O_{K_\infty})$ depends on the choices of $T$ and $e$, its image does not, and we denote this image by $\S_K$. In \secsymb\ref{ssec:A_K}, we extend $\S_K$ to a prism $\A_K^+$ in $(\O_K)_{\Prism_L}$ using the theory of imperfect norm fields. In \secsymb\ref{ssec:phi_q-Gamma-modules} we recall the definition of Lubin-Tate $(\varphi_q, \Gamma)$-modules. Later, in \secsymb\ref{ssec:phi-G-modules-F-crystals} we will see that $\A_K^+$ and $\Ainf(\O_{K_\infty})$ are ``large enough'' that Laurent $F$-crystals over $(\O_K)_{\Prism_L}$ are equivalent to $(\varphi_q, \Gamma_K)$-modules over $\A_K$ or $W_L(K_\infty^\flat)$.

\subsection{Construction of \texorpdfstring{$\S_K$}{S\_K}}\label{ssec:S_K}

In this subsection we construct prisms $(\S_K,(q_n(\omega)))\in (W_L(k))_{\Prism_L}$ for $n\ge 1$; when $n = 1$, this prism plays the same role that the $q$-de Rham prism $(\Z_p\ps{q - 1}, \tfrac{q^p - 1}{q - 1})$ plays in the cyclotomic theory. By construction we will have that $\S_K$ is a sub-$\delta_L$-algebra of $W_L(\O_{K_\infty}^\flat)$, so that the map $\phi^{-1}:\S_K\rightarrow W_L(\O_{K_\infty}^\flat)$ maps sense, for $\phi = \phi_{W_L(\O_{K_\infty}^\flat)}$. We will also show that $\O_{K_\infty}$ is a perfectoid $\O_L$-algebra and that the map $\phi^{-n}: (\S_K,(q_n(\omega)))\rightarrow (W_L(\O_{K_\infty}^\flat),\ker\theta)$ is a map of prisms.

Fix all notation as at the start of this section, and equip $W_L(k)\ps{T}\cong W_L(k)\otimes_{\O_L}\O_L\ps{T}$ with the lift $\phi$ of $q$-Frobenius coming from the natural $\phi$-actions on $W_L(k)$ and $\O_L\ps{T}\cong \O_\G$; explicitly this is given by
	\[f(T)\mapsto f^{\phi_{W_L(k)}}([\pi](T))\]
where $f^{\phi_{W_L(k)}}$ denotes the coefficient-wise action on $f\in W_L(k)\ps{T}$. For $n\ge 1$ set
\[q_n(T) = \frac{[\pi^n](T)}{[\pi^{n - 1}](T)}\in \O_L\ps{T}.\]
\begin{lem}\label{lem:S_K-L-typical-prism}
	$\left(W_L(k)\ps{T}, (q_n(T))\right)$ is a $L$-typical prism in $(W_L(k))_{\Prism_L}$ for every $n\ge1$.
\end{lem}
\begin{proof}
	For $(\pi,q_n(T))$-adic completeness use that $q_n(T)\equiv T^{q^n - q^{n - 1}}\pmod{\pi}$ and the $(\pi,T)$-adic completeness of $W_L(k)\ps{T}$. To see that $\pi\in (q_n(T),\phi(q_n(T))) = (q_n(T), q_{n + 1}(T))$, note that $q_1(T) = \frac{[\pi](T)}{T}\equiv \pi\pmod{T}$ so that
	\[\pi = q_1(T) + Tf(T)\]
	for some $f(T)\in \O_L\ps{T}$; applying $\phi^n$ then gives
	\[\pi = q_{n + 1}(T) + [\pi^n](T)f([\pi^n](T)) = q_{n + 1}(T) + q_n(T)[\pi^{n - 1}](T)f([\pi^n](T))\in (q_n(T), q_{n + 1}(T)).\]
\end{proof}

We now invoke construction~\ref{const:algebraic-construction} to produce a map $W_L(k)\ps{T}\rightarrow W_L(\O_{K_\infty}^\flat)$ of $\delta_L$-algebras. Since we have not yet shown that $\O_{K_\infty}$ is a perfectoid $\O_L$-algebra, we cannot yet use construction~\ref{const:prismatic-construction}, but once we have shown that $\O_{K_\infty}$ is perfectoid the two constructions amount to the same thing.

Indeed, apply construction~\ref{const:algebraic-construction} with $A = W_L(k)\ps{T}$, $R = \O_{K_\infty}$, and $\iota_n:W_L(k)\ps{T}\rightarrow R$ defined by sending $f(T)$ to $f^{\phi_{W_L(k)}^{-n}}(e_n)$. Note that the system $(\iota_n)_n$ is $\phi$-compatible since  
\[\iota_{n + 1}(\phi(f)) = \iota_{n + 1}(f^{\phi_{W_L(k)}}([\pi](T))) = f^{\phi_{W_L(k)}^{-n}}([\pi](e_{n+1})) = \iota_n(f),\]
where we've used that $\phi$ acts as the identity on $\O_L$ and $[\pi](e_{n + 1}) = e_n$. This gives us a map $\delta_L$-algebra map $\iota:W_L(k)\ps{T}\rightarrow W_L(\O_{K_\infty}^\flat)$ lifting the map $k\ps{T}\rightarrow \O_K^\flat$ given by
\[T\mapsto \overline{\omega} := (\dots, \overline{e_2}, \overline{e_1}, 0)\]
where $\overline{e_n}\in \O_{K_\infty}/\pi$ is the mod $\pi$ reduction of $e_n\in \O_{K_n}$ for $n\ge 0$. Let $\omega = \iota(T)\in W_L(\O_{K_\infty}^\flat)$ denote the given lift of $\overline{\omega}$, and write $\S_K = W_L(k)\ps{\omega}\subseteq W_L(\O_{K_\infty})$ for the image of $\iota$. By lemma~\ref{lem:S_K-L-typical-prism}, $(S_K,q_n(\omega))$ is a $L$-typical prism for every $n\ge 1$. 

\begin{rmk}
	When $n = 1$ and $\G = \mu_{p^\infty}$, we see that $q_1(T) = \frac{(1 + T)^p - 1}{T}$. After a change of variables $T\mapsto T - 1$, we thus get the $q$-de Rham prism.
\end{rmk}

\begin{rmk}\label{rmk:change-by-unit}
	A different choice of coordinate on $\G$ amounts to changing $T$ by a unit in $\O_L\ps{T}$, and a different choice of basis for the Tate module of $\G$ corresponds to multiplying $e$ by a unit in $\O_L$. Hence changing $T$ and $e$ results in changing $\omega$ by a unit but does not change the image $\S_K$ of $W_L(k)\ps{T}$ in $\Ainf(\O_{K_\infty})$.
\end{rmk}

\begin{lem}\label{lem:K_infty-perfectoid}
	$\O_{K_\infty}$ is a perfectoid $\O_L$-algebra, and 
	\[(\S_K, (q_n(\omega)))\stackrel{\phi^{-n}}{\longrightarrow}(\Ainf(\O_{K_\infty}), \ker\theta)\]
	is a map of prisms for every $n\ge 1$.
\end{lem}
\begin{proof}
	Note first that by proposition~\ref{prop:delta-alg-map-computation}, we have $\phi^{-n}(q_n(\omega))\in \ker\theta$ since
	\[\theta(\phi^{-n}(q_n(\omega))) = (\theta\circ\phi^{-n}\circ\iota)(q_n(T)) = \iota_n(q_n(T)) = \frac{[\pi^n](e_n)}{[\pi^{n - 1}](e_n)} = 0.\]
	Thus, we'll be done if we show that $(W_L(\O_{K_\infty}^\flat), \ker\theta)$ is a prism (equivalently, that $\O_{K_\infty}$ is a perfectoid $\O_L$-algebra). One way to procede would be to use a rigidity result like \cite[Lemma~3.6]{BS}. Instead, we'll use proposition~\ref{prop:perfectoid-charactierzation}; $O_{K_\infty}$ clearly satisfies conditions (1), (2), and (4'), so it sufices to show that it satifies condition (2) as well.

	Let $d = \phi^{-n}(q_n(\omega))\in \ker\theta$. Following the proof of proposition~\ref{prop:perfectoid-charactierzation}, we guess that $\xi = \theta(\phi^{-1}(d))\in \O_{K_\infty}$ satisfies $\xi^d = \pi u$ for a unit $u\in \O_{K_\infty}^\times$. Indeed, we have
	\[\xi^d = \theta(\phi^{-q}(d^q)) = \theta(d - \pi\delta_L(\phi^{-1}(d))) = \pi\theta(-\delta_L(\phi^{-1}(d))),\]
	and since $q_n(T)\in \S_K$ is distinguished, so is $\phi^{-1}(d) = \phi^{-n - 1}\iota(q_n(T))\in W_L(\O_{K_\infty}^\flat)$.
\end{proof}

\begin{rmk}
	Though we will not use this fact, we note that in this setting, there is an analytic way to construct the map $\iota:W_L(k)\ps{T}\rightarrow W_L(\O_{K_\infty}^\flat)$. Namely, following \cite[lemma~1.2]{KR}, we let $\hat\omega$ be any lift of $\overline{\omega} = (\dots,\overline{e_2},\overline{e_1},0)\in \O_{K_\infty}^\flat$, and set
	\[\omega = \lim_{n\rightarrow\infty}[\pi^n](\phi_{W_L(\O_{K_\infty}^\flat)}^{-n}(\hat\omega)).\]
	Then $\phi(\omega) = [\pi](\omega)$, so that
	\begin{align*}
		W_L(k)\ps{T}&\rightarrow W_L(\O_{K_\infty}^\flat) \\
		T&\mapsto \omega
	\end{align*}
	is a $\delta_L$-algebra map lifting the $\O_L$-algebra map $W_L(k)\ps{T}\rightarrow \O_{K_\infty}^\flat$ via $T\mapsto \overline{\omega}$, hence it coincides with $\iota$ by lemma~\ref{lem:W-delta_pi-functor}.
\end{rmk}

\subsection{Extension to \texorpdfstring{$\A_K^+$}{A\_K\^+}}\label{ssec:A_K}

The prisms $(\S_K,(q_n(\omega)))$ from \secsymb\ref{ssec:S_K} can be viewed as objects in $(W_L(k)[e_n])_{\Prism_L}$. However, there is ramification in $K/L$ outside of the ramification in $L_\infty/L$ (i.e. $\O_K\not\subseteq \bigcup_{n\ge 0}W_L(k)[e_n]$) then $(\S_K,(q_n(\omega)))$ will never be a prism over $\Spf \O_K$. In this section, we will extend $\S_K$ to a larger sub-$\delta_L$-algebra $\A_K^+$ of $\Ainf(\O_{K_\infty})$ which is a prism over $\Spf\O_K$. The key point is that the formation of $\S_K$ is insensitive to taking ramified extensions of $K$; we capture ramification coming from the tower $L_\infty/L$ by our choice of $n$ (since $\S_K/q_n(\omega)\cong W_L(k)[e_n]$), but capturing the rest of the ramification in $K/W_L(k)[1/\pi]$ requires Fontaine and Wintenberger's theory of imperfect norm fields.

Let
\[\E_K^+ = \left\{(\alpha_n)_n\in \plim_{\varphi_q} \O_{K_\infty}/e_1 = \O_{K_\infty}^\flat : \alpha_n\in \O_{K_n}/e_1 \text{ for }n\gg 0\right\}\subseteq \O_{K_\infty}^\flat,\]
so that $\overline{\omega} = (\dots, \overline{e_2}, \overline{e_1}, 0)\in \E_K^+$. We recall some facts from the theory of norm fields \cite{W}.
\begin{prop}~\label{prop:norm-field-facts}
	\begin{enumerate}
		\item[(1)] $\E_K^+$ is a complete discrete valuation ring with fraction field $\E_K:= \E_K^+[1/\overline{\omega}]\subseteq K_\infty^\flat$.
		\item[(2)] If $K/L$ is unramified, then $\E_K^+ = k\ps{\overline{\omega}}$. In general, $\E_K$ is a totally ramified extension of $\E_{W_L(k)[1/\pi]}$ of degree $[K_n:W_L(k)[e_n][1/\pi]]$ for $n$ large enough.
		\item[(3)] The completed perfection $(\varinjlim_{\varphi_q} \E_K^+)^\wedge_{(\overline{\omega})}$ of $\E_K^+$ is $\O_{K_\infty}^\flat$.
		\item[(4)] There is an equivalence of Galois categories
		\[\left\{\text{finite extensions of }\bigcup_{n\ge 1}L_n\text{ in }\overline{L}\right\}\simeq \left\{\text{finite seperable extensions of }\E_L\text{ in }K_\infty^\flat\right\}\]
		where, given a finite subextension $M/\bigcup_{n\ge 1}L_n$ of $\overline{L}$, the functor from the left to the right is given by selecting any finite $M'/L$ with $\bigcup_n M_n' = M$ and sending $M$ to $\E_{M'}$.
	\end{enumerate}
\end{prop}

We would like to form Cohen rings $\A_K$ for the fields $\E_K$ in characteristic $p$. For $K = W_L(k)[1/\pi]$ unramified over $L$, we write
\[\A_K = \S_K[1/\omega]^\wedge_{(\pi)}\subseteq W_L(K_{\infty}^\flat)\]
for the $\pi$-adic completion of $\S_K[1/\omega]\cong W_L(k)\ps{T}[1/T]$. Then $\A_K$ is a complete discrete valuation ring in characteristic zero with uniformizer $\pi$, and by proposition~\ref{prop:norm-field-facts} $\A_K$ has residue field $\E_K$. When $K/L$ is possibly ramified, Hensel's lemma allows us to lift the extension $\E_K$ of $\E_{W_L(k)[1/\pi]}\cong k\ls{T}$ to an unramified extension $\A_K$ of $\A_{W_L(k)[1/\pi]}\cong W_L(k)\ps{T}[1/T]^\wedge$ inside of $W_L(K_\infty^\flat)$.
\begin{center}
	\begin{tikzcd}
		W_L(k)\ps{T}[1/T]^\wedge\arrow[r,"\iota","\sim"']\ar[twoheadrightarrow]{d} &\A_{W_L(k)[1/\pi]}\ar[hook]{r}\ar[twoheadrightarrow]{d} &\A_K\ar[hook]{r}\ar[twoheadrightarrow]{d}&W_L(K_\infty^\flat)\ar[twoheadrightarrow]{d} \\
		k\ls{T}\ar[r,"\overline{\iota}", "\sim"'] &\E_{W_L(k)[1/\pi]}\ar[hook]{r} & \E_K\ar[hook]{r} & K_\infty^\flat
	\end{tikzcd}
\end{center}
By construction, $\A_K$ is stable under $\phi_{W_L(K_\infty^\flat)}$ (since $\phi(a)\mod{\pi} = \overline{a}^q\in \E_K$ for any $a\in \A_K$). Thus we set
\[\A_K^+ = \A_K\cap W_L(\O_{K_\infty}^\flat),\]
which is $\phi$-stable as well. Since $\A_K^+$ is $\pi$-torsionfree, this gives it a $\delta_L$-algebra structure. Note that when $K/L$ is unramified, we have $\A_K^+ = \S_K$.

\begin{rmk}
	Instead of forming $\A_K$ by lifting the extension $\E_K/E_{W_L(k)[1/\pi]}$ over $\A_{W_L(k)[1/\pi]}$, we could have instead lifted the extension $\E_K/\E_L$ over $\A_L$. This would have amounted to the same thing. We also note that the $\phi$-action on $\A_K$ above clearly coincides with the one induced by lifting $\varphi_q:\E_K\rightarrow \E_K$ via Hensel's lemma (and using that $\A_{W_L(k)[1/\pi]}$ is $\phi$-stable by construction).
\end{rmk}
\begin{rmk}
	Note that $\A_K = \A_K^+[1/q_n(\omega)]^\wedge_{(\pi)}$ since $q_n(\omega)\equiv \omega^{q^{n - 1}(q - 1)}\pmod{\pi}$, so that after $\pi$-adically completing, inverting $\omega$ has the same effect as inverting $q_n(\omega)$. 
\end{rmk}

\begin{lem}~\label{lem:prism-maps-flat}
\begin{enumerate}
	\item[(1)] If $A\rightarrow B$ is a map of $\pi$-adically complete $\pi$-torsion free rings with $A$ noetherian and $A/\pi\rightarrow B/\pi$ is flat, then $A\rightarrow B$ is flat as well.
	\item[(2)] The maps
	\[\S_K\hookrightarrow \A_K^+,\quad\quad \A_K^+\hookrightarrow \Ainf(\O_{K_\infty}),\quad\text{and}\quad\quad \phi:\A_K^+\rightarrow \A_K^+\]
	are all faithfully flat.
	\item[(3)] $\S_K/q_n(\omega)$ and $\A_K^+/q_n(\omega)$ are $\pi$-torsion free. $(\S_K,(q_n(\omega)))$ and $(\A_K^+, (q_n(\omega)))$ are bounded.
\end{enumerate}
\end{lem}
\begin{proof}
	Part (1) is \cite[remark~4.31]{BMS} with $p$ replaced by $\pi$; the proof remains the same. The flatness in (2) follows from (1) since the mod $\pi$ reductions of the given maps are
	\[k\ps{T}\hookrightarrow \E_K^+,\quad\quad \E_K^+\hookrightarrow \O_{K_\infty}^\flat,\quad\text{and}\quad\quad \varphi_q:\E_K^+\rightarrow \E_K^+\]
	which are injective maps from discrete valuation rings to integral domains, hence flat. Faithful flatness follows since $\overline{\omega}$ is not a unit in $\E_K^+$ or $\O_{K_\infty}^\flat$. For (3), we have that $\S_K/q_n(\omega)\cong W_L(k)\ps{T}/q_n(T)\cong W_L(k)[e_n]$ is an integral domain hence $\pi$-torsion free. By part (2) we have that $\S_K/q_n(\omega)\rightarrow \A_K^+/q_n(\omega)$ is flat, so $\A_K^+/q_n(\omega)$ is $\pi$-torsion free as well.
\end{proof}

It follows immediately from lemma~\ref{lem:S_K-L-typical-prism} that $(\A_K^+,(q_n(\omega)))$ is a $L$-typical prism for every $n\ge 1$. Moreover, just as $\E_K^+$ can be viewed as a deperfection of $\O_{K_\infty}^\flat$, we have that the prism $(\A_K^+, (q_n(\omega)))$ can be viewed as a deperfection of the perfect prism $(\Ainf(\O_{K_\infty}), \ker\theta)$. 
\begin{prop}\label{prop:A_K^+-perfection}
	Let $(A_{\perf}, IA_{\perf})$ be the perfection of $(\A_K^+, (q_n(\omega)))$ as in proposition~\ref{prop:prism-perfection}. Then $A_{\perf}\cong \Ainf(\O_{K_\infty})$. The natural map $\A_K^+\rightarrow A_{\perf}\cong \Ainf(\O_{K_\infty})$ is the usual inclusion, i.e. the map on the left in the following commutative diagram.
	\begin{center}
	\begin{tikzcd}
		(\Ainf(\O_{K_\infty}),(q_n(\omega)))\arrow[rr, "\phi^{-n}", "\sim"'] && (\Ainf(\O_{K_\infty}),\ker\theta) \\
		&(\A_K^+,(q_n(\omega)))\ar{ul}{\subseteq}\arrow[ur,"\phi^{-n}"']
	\end{tikzcd}
	\end{center}
\end{prop}
\begin{proof}
	By proposition~\ref{prop:perfect-d-algs-equiv}, it suffices to show that $A_{\perf}/\pi\cong \O_{K_\infty}^\flat$. Indeed, we have
	\[A_{\perf}/\pi = (\varinjlim_\phi \A_K^+)^\wedge_{(q_n(\omega))}/\pi\cong (\varinjlim_{\varphi_q} \E_K^+)^\wedge_{(\overline{\omega})} = \O_{K_\infty}^\flat\]
	since $q_n(\omega)\equiv \omega^{q^n - q^{n - 1}}\pmod{\pi}$ and modding out by $\pi$ commutes with the colimit and $(\pi,q_n(\omega))$-adic completion.
\end{proof}

\begin{cor}\label{cor:A_K-O_K}
	For $n\gg 0$ we have structure maps $\O_K\rightarrow \A_K^+/q_n(\omega)$ such that the maps
	\[(\A_K^+, (q_n(\omega)))\stackrel{\phi^{-n}}{\longrightarrow} (\Ainf(\O_{K_\infty}), \ker\theta)\]
	are morphisms in $(\O_K)_{\Prism_{\pi}}$.
\end{cor}
We will give two proofs, the first an abstract argument following \cite[prop~2.19]{Wu} and the second a more concrete argument involving norm fields.
\begin{proof}[Proof 1]
	By proposition~\ref{prop:A_K^+-perfection} we have an isomorphism
	\[\left(\varinjlim_\phi \A_K^+\right)^\wedge_{(\pi, q_1(\omega))}/q_1(\omega)\stackrel{\phi^{-1}}{\longrightarrow}\Ainf(\O_{K_\infty})/\ker\theta = \O_{K_\infty}.\]
	This isomorphism can be rewritten as
	\[\left(\varinjlim_\phi \A_K^+/q_n(\omega)\right)^\wedge_{(p)}\stackrel{\sim}{\longrightarrow} \left(\bigcup_{n\ge 1} \O_{K_n}\right)^\wedge_{(p)}.\]
	Using that $\varinjlim \A_K^+/q_n(\omega)$ and $\bigcup \O_{K_n}$ are integral over $W_L(k)$ and that there are no integral extensions between $\bigcup \O_{K_n}$ and its completion $\O_{K_\infty}$ (by Krasner's lemma applied to the Henselian ring $\bigcup \O_{K_n}$), we conclude that there is a short exact sequence
	\[0\longrightarrow \varinjlim_\phi \A_K^+/q_n(\omega)\longrightarrow \bigcup \O_{K_n}\longrightarrow M\longrightarrow 0\]
	with $M$ $\pi$-torsion and $M_{(\pi)}^\wedge = 0$. Moreover, since $\varinjlim \A_K^+/q_n(\omega)$ contains the subring $\varinjlim \S_K/q_n(\omega)\cong \bigcup W_L(k)[e_n]$ over which $\bigcup \O_{K_n}$ is finite, we have that $M$ is $\pi$-adically complete, so that $M = M^\wedge = 0$.

	Moreover, since $\S_K/q_n(\omega)\stackrel{\phi}{\rightarrow}\S_K/q_{n + 1}(\omega)$ identifies with the inclusion
	\[W_L(k)[e_n]\hookrightarrow W_L(k)[e_{n + 1}]\]
	and $\S_K\hookrightarrow \A_K^+$ is flat by lemma~\ref{lem:prism-maps-flat}, the transition maps in the direct limit are injective as well. All together, this gives
	\[\bigcup_{n\ge 1} \A_K^+/q_n(\omega)\cong \bigcup_{n\ge 1}\O_{K_n}\supseteq \O_K.\]
	As $\O_K$ is finite over $W_L(k)$ and the left-hand side is an increasing union of $W_L(k)$-modules, we get maps $\O_K\rightarrow \A_K^+/q_n(\omega)$ for $n\gg 0$. These maps commute with $\phi^{-n}:\A_K^+\rightarrow \Ainf(\O_{K_\infty})$ by construction.
\end{proof}
\begin{proof}[Proof 2.]
	To simplify notation, set $F = W_L(k)[1/\pi]$. Let 
	\[\overline{\omega}_K = (\overline{\pi}_n)_n\in \plim_{\varphi_q} \O_{K_\infty}/e_1 = \O_{K_\infty}^\flat\] 
	be a uniformizer of $\E_K^+$, so that $\overline{\pi}_n\in \O_{K_n}/e_1$ for $n\gg 0$ and $\E_K^+ = k\ps{\overline{\omega}_K}$. Let $P(W,T)\in k\ps{W}[T]$ so that $P(\overline{\omega},T)\in k\ps{\overline{\omega}}[T] = \E_F^+[T]$ is the minimal polynomial of $\overline{\omega}_K$ over $\E_F$. As explained above, $\E_K/\E_F$ is a totally ramified extension of degree $d = [K_\infty:F_\infty]$, so $P(\overline{\omega}, T)$ is a degree $d$ Eisenstein polynomial. Since $\overline{\omega} = (\overline{e_n})_n$, it follows that $P(\overline{e_n}, \overline{\pi_n}) = 0\in \O_{K_n}/e_1$ for $n\gg 0$.

	Let $\hat P(W,T)\in \O_F\ps{W}[T]$ be a lift of $P$. Using an argument involving Lang's refinement of Hensel's lemma, it is shown in \cite[3.2.5]{W} (or see also \cite[pf of prop~13.4.4]{brinnonconrad}) that $P(e_n,T)$ has $d$ distinct roots $\{\pi_{n,1},\dots, \pi_{n,d}\}$ in $\O_{\overline{K}}$; one of these roots, call it $\pi_n$, is a lift of $\overline{\pi}_n$. Moreover, using that the roots of $P(\overline{\omega},T)$ in $\O_{K_\infty}^\flat$ are distinct, one can show that $\pi_{n,1},\dots, \pi_{n,d}$ are distinct mod $e_1$ for $n\gg 0$. On the other hand, since $\pi_{n + 1}^q\equiv \pi_n\pmod{e_1}$ for $n\gg 0$, Krasner's lemma shows that for $n\gg 0$ we have $\pi_n\in F_n(\pi_{n + 1})$. 

	Now, set $K_n' = F_n(\pi_n)$. For some $N\gg 0$ and all $n\ge N$, we have that $K_n'\subseteq K_{n + 1}'$. This gives us the following diagram of field extensions, with degrees indicated.
	\begin{center}
		\begin{tikzcd}
				& K_{n + 1}' \\
			F_{n + 1}\ar[dash]{ur}{d} \\ \\
				& K_n'\ar[dash]{uuu}{} \\
			F_n\ar[dash]{uuu}{q}\ar[dash]{ur}{d}
		\end{tikzcd}
	\end{center}
	This implies that $K_{n + 1}' = K_n'F_{n + 1}$, and thus that $K_n' = K_N'F_n$ for all $n\ge N$. We also have that $\E_{K_N'} = \E_K$, since they are both degree $d = [K_\infty:F_\infty] = [K_{N,\infty}':F_\infty]$ extensions of $\E_F$ and $\E_K = k((\overline{\omega}_K))\subseteq \E_{K'}$ by construction. Thus by proposition~\ref{prop:norm-field-facts}(4), we get that $\bigcup_n K_{N,n}' = \bigcup_n K_n$, so that for $n\gg 0$, we have $K\subseteq K_{N,n}'$. Hence for $n\gg 0$,
	\[\O_K\subseteq \O_{K_{N,n}'} = \O_F[e_n][\pi_n] = \O_F\ps{\omega}[T]/(q_n(\omega),\hat P(\omega, T)) = \A_K/(q_n(\omega))\]
	as desired.

	Using the formula $\theta(\phi^{-n}(\omega)) = e_n$ (which follows from proposition~\ref{prop:delta-alg-map-computation}), we can trace the inclusion above through the map $\phi^{-n}:\A_K/(q_n(\omega))\rightarrow \Ainf(\O_{K_\infty})/\ker\theta$ to find that it coincides with $\O_K\subseteq \O_{K_\infty}\cong \Ainf(\O_{K_\infty})/\ker\theta$. 
\end{proof}

\subsection{\texorpdfstring{$\Gamma_K$}{Gamma\_K}-actions and \'etale \texorpdfstring{$(\varphi_q, \Gamma)$}{(phi, Gamma)}-modules}\label{ssec:phi_q-Gamma-modules}

In this section we summarize the main results in the theory of Lubin-Tate $(\varphi_q,\Gamma)$-modules. These results will be recovered as special cases of the results in \secsymb\ref{sec:F-crystals}.

\begin{defn}
	\textit{A $\varphi_q$-module over $\A_K$} is a finite flat $\A_K$ module $M$ equipped with a $\phi_{\A_K}$-semilinear endomorphism $\phi_M:M\rightarrow M$. It is called \textit{\'etale} if the $\A_K$-linear map
	\begin{align*}
		\phi^*M := \A_K\otimes_{\phi, \A_K}M&\longrightarrow M \\
		a\otimes m&\mapsto a\phi_M(m)
	\end{align*}
	is an isomorphism. When equipped with $\A_K$-module maps that commute with the $\phi_M$'s, these form categories $\Mod_{\A_K}^{\varphi_q}$ and $\Mod_{\A_K}^{\varphi_q,et}$. We similarly define $\varphi_q$-modules over $W_L(K_\infty^\flat)$ and the categories $\Mod_{W_L(K_\infty^\flat)}^{\varphi_q}$ and $\Mod_{W_L(K_\infty^\flat)}^{\varphi_q,et}$.
\end{defn}
By a result of Kisin-Ren \cite{KR} (and Fontaine \cite{F} in the cyclotomic case), we have that \'etale $\varphi_q$-modules are equivalent to the category $\Rep_{\O_L}(G_{K_\infty})$ of continuous finite free $G_{K_\infty} = \Gal(\overline{K}/K_\infty)$-representations over $\O_L$. In more detail, proposition~\ref{prop:norm-field-facts}(4) implies that $\E := \bigcup_{K/L} \E_K$ is the seperable closure of $\E_L$, and that $\Gal(\E/\E_K) = G_{K_\infty}$. It follows that $\A := \left(\bigcup_{K/L}\A_K\right)^\wedge$ is the completion of the maximal unramified extension of $\A_L$; $\A$ thus inherits a $\Gal(\E/\E_L) = G_{L_\infty}$-action with $\A^{G_{K_\infty}} = \A_K$. Moreover, $\A\subseteq W_L(K_\infty^\flat)$ has a $\phi$-action. The key theorem is as follows.
\begin{thm}(cf. \cite[Theorem~1.6]{KR})~\label{thm:kisin-ren-equiv}
	The functors
	\begin{center}
	\begin{tikzcd}
		\Mod_{\A_K}^{\varphi_q,et}\ar[bend left]{rr}{M\mapsto (M\otimes_{\A_K}\A)^{\phi = 1}} &&\Rep_{\O_L}(G_{K_\infty})\ar[bend left]{ll}{(T\otimes_{\O_L}\A)^{G_{K_\infty}}\mapsfrom T}
	\end{tikzcd}
	\end{center}
	form an equivalence of exact tensor categories.
\end{thm}

We make two observations about Theorem~\ref{thm:kisin-ren-equiv}. First, that the base of the $\varphi_q$-modules is $\A_K$. In fact, this is a red herring: base change induces an equivalence of categories
\begin{align*}
	\Mod_{\A_K}^{\varphi_q,et}&\stackrel{\sim}{\longrightarrow}\Mod_{W_L(K_\infty^\flat)}^{\varphi_q,et}\\
	M&\mapsto M\otimes_{\A_K}W_L(K_\infty^\flat)
\end{align*}
so that theorem would remain true with $W_L(K_\infty^\flat)$ replacing $\A_K$ (and $W_L(\overline{K}^\flat)$ replacing $\A$). (This result is due to Fontaine \cite{F} in the cyclotomic case, but as far as the author is aware has not yet appeared in the literature in general; it will follow from proposition~\ref{prop:perfection-base-change} below.)

Our second observation is that we would like to descend the equivalence to the full category $\Rep_{\O_L}(G_K)$ of continuous finite free $G_K$-representations over $\O_L$. Indeed, this is not hard to do, and involves picking up a semilinear action of $\Gamma_K = \Gal(K_\infty/K)$. Before stating the result, we first explain the $\Gamma_K$ actions on the rings $\S_K$, $\A_K$, and $W_L(K_\infty^\flat)$.

Equip $W_L(k)\ps{T}$ with the $\Gamma_K$-action where $\sigma\in \Gamma_K$ acts by $f(T)\mapsto f([\chi_\G(\sigma)](T))$, and equip $W_L(\O_{K_\infty}^\flat)$ with the natural $\Gamma_K$-action (coming from the $\Gamma_K$-action on $K_\infty^\flat)$ and the functoriality of $W_L$). By the definition of the Lubin-Tate character $\chi_\G$ we have
\[[\chi_\G(\sigma)](\overline{\omega}) = (\overline{e_n^\sigma})_n = \overline{\omega}^\sigma\]
so that $W_L(k)\ps{T}\twoheadrightarrow k\ps{T}\stackrel{\overline{\iota}}{\rightarrow} \O_{K_\infty}^\flat$ is $\Gamma_K$-equivariant. Thus $\iota$ is $\Gamma_K$-equivariant as well by naturality, and the $\Gamma_K$-actions on $\S_K$ induced by $\iota$ and $W_L(K_\infty^\flat)$ coincide. By the uniqueness of lifts given by Hensel's lemma, we further have that the $\Gamma_K$-action on $\A_K$ induced by viewing it as a subring of $W_L(K_\infty^\flat)$ coincides with the $\Gamma_K$-action defined by lifting the $\Gamma_K$-action on $\E_K$. Note also that all of these $\Gamma_K$-actions commute with the $\phi$-actions, because this is true for $W_L(K_\infty^\flat)$. This can also be seen directly for $W_L(k)\ps{T}$ using properties of Lubin-Tate formal $\O_L$-modules:
\[\phi(f)^{\sigma}(T) = f([\pi\chi_\G(\sigma)](T)) = f([\chi_\G(\sigma)\pi](T)) = \phi(f^{\sigma})(T).\]

We can also view $\Gamma_K$ as acting on the corresponding prisms.
\begin{prop}\label{prop:Ainf-automorphisms}
	$\Gamma_K$ acts via automorphisms on the $L$-typical prisms $(\A_K^+, (q_n(\omega)))$ and $(\Ainf(\O_{K_\infty}), \ker\theta)$. Moreover, we have that
	\[\Aut_{(\O_K)_{\Prism_L}}(\A_K^+, (q_n(\omega)))\cong \Aut_{(\O_K)_{\Prism_L}}(\Ainf(\O_{K_\infty}), \ker\theta)\cong \Gamma_K\]
	if $n$ is large enough that $(\A_K^+,(q_n(\omega)))\in (\O_K)_{\Prism_L}$.
\end{prop}
\begin{proof}
	Since any $\sigma\in \Gamma_K$ commutes with $\phi$, we know that $\sigma$ gives a map of $\delta_L$-algebras. Aditionally, since $q_n(\omega)^{\sigma} = [\chi_\G(\sigma)](q_n(\omega))$ and 
	\[[\chi_\G(\sigma)](T) = \chi_\G(\sigma)T +\text{ higher order terms},\] 
	we have that any $\sigma\in \Gamma_K$ preserves $(q_n(\omega))$, and thus gives an automorphism of $(\A_K^+, (q_n(\omega)))$. If $n$ is large enough that $(\A_K^+,(q_n(\omega)))\in (\O_K)_{\Prism_L}$ then $\sigma$ respects the structure map $\O_K\rightarrow \A_K^+/q_n(\omega)$ as well, so that 
	\[\Gamma_K\hookrightarrow \Aut_{(\O_K)_{\Prism_L}}(\A_K^+,(q_n(\omega))).\]
	Moreover, we see that any automorphism of $(\A_K^+, (q_n(\omega)))$ is automatically $(\pi,q_n(\omega))$-adically continuous and $\phi$-equivariant, hence extends to an automorphism of the perfection $(\A_K^+, (q_n(\omega)))_{\perf}\cong (\Ainf(\O_{K_\infty}),\ker\theta)$ by proposition~\ref{prop:A_K^+-perfection}. But proposition~\ref{prop:perfect-prism-perfectoid-equiv}, we have that
	\[\Aut_{(\O_K)_{\Prism_L}}(\Ainf(\O_{K_\infty}),\ker\theta)\cong \Aut(\O_{K_\infty}/\O_K) = \Gamma_K.\]
	Thus we've shown
	\[\Gamma_K\hookrightarrow \Aut_{(\O_K)_{\Prism_L}}(\A_K^+,(q_n(\omega)))\hookrightarrow \Aut_{(\O_K)_{\Prism_L}}(\Ainf(\O_{K_\infty}),\ker\theta) \cong \Gamma_K\]
	which gives the result.
\end{proof}

Descending from $\Rep_{\O_L}(G_{K_\infty})$ to $\Rep_{\O_L}(G_K)$ involves picking up a $\Gamma_K$-action.
\begin{defn}
	A \textit{$(\varphi_q, \Gamma)$-module over $\A_K$} is a $\varphi_q$-module $M$ over $\A_K$ with a semilinear $\Gamma_K$-action which commutes with the $\phi$-action. It is \textit{\'etale} if $M$ is \'etale as a $\varphi_q$ module. These form categories $\Mod_{\A_K}^{\varphi_q,\Gamma}$ and $\Mod_{\A_K}^{\varphi_q,\Gamma,et}$. We similarly define $(\varphi_q,\Gamma)$-modules over $W_L(K_\infty^\flat)$.
\end{defn}
The equivalence of Theorem~\ref{thm:kisin-ren-equiv} extends to $(\varphi_q, \Gamma)$-modules. So in summary, we have the following inclusions and equivalences among exact tensor categories.
\begin{center}
	\begin{tikzcd}
		\Rep_{\Z_p}(G_K)\ar{r}{\cong}\ar[hook]{d} &\Mod_{\A_K}^{\varphi,\Gamma,et}\ar{r}{\cong}\ar[hook]{d} &\Mod_{W(K_\infty^\flat)}^{\varphi,\Gamma,et}\ar[hook]{d} \\
		\Rep_{\Z_p}(G_{K_\infty})\ar{r}{\cong} &\Mod_{\A_K}^{\varphi,et}\ar{r}{\cong} &\Mod_{W(K_\infty^\flat)}^{\varphi,et}
	\end{tikzcd}
\end{center}

\subsection{The prismatic logarithm for \texorpdfstring{$\S_L$}{S\_L}}\label{ssec:prismatic-logarithm}
For convenience, throughout this section let $(A,I) = (\A_L^+,(q_n(\omega))) = (\S_L, (q_n(\omega)))\cong (\O_L\ps{T}, (q_n(T)))$ be the prism of \secsymb~\ref{ssec:S_K}. We will construct a map $\log_\Prism$ from a certain subset $I_{\phi = [\pi]}$ of $I$ to the Breuil-Kisin twist $A\{1\}$ of $A$. Heuristically, we can think of $\log_\Prism$ as being given by $\log_{\Prism}(u) = ``\lim_{n\rightarrow\infty}\frac{[\pi^n](u)}{\pi^n}"$. We will further see that $\log_\Prism$ is $\O_L$-linear, where $I_{\phi = [\pi]}$ is viewed as an $\O_L$-module via the Lubin-Tate formal group law $\G$. 
\begin{rmk}
	In the cyclotomic case $\G = \mu_{p^\infty}$, our $\log_{\Prism}$ coincides with the map $u\mapsto \log_\Prism(1 + u)$ of \cite[\secsymb2]{BL}. In that setting, $\log_{\Prism}(1 + u) = ``\lim_{n\rightarrow\infty}\frac{u^{p^n} - 1}{p^n}"$, which is analogous to the classical formula $\log(1 + x) = \lim_{\alpha\rightarrow 0}\frac{x^\alpha - 1}{\alpha}$.
\end{rmk}

For this paragraph only, let $(A,I)$ be an arbitrary bounded $L$-typical prism. Informally, we define
\[A\{1\} = \bigotimes_{n = 0}^\infty(\phi^n)^* I.\]
More precisely, for $n\ge 1$ set $I_n$ to be the product $\prod_{i = 0}^{n - 1}\phi^i(I)$ as an ideal of $A$. Note that $I_n\equiv I^{\frac{q^n - 1}{q - 1}}\pmod{\pi}$. Thus, since $A$ is bounded and $(\pi,I)$-adically complete we have $\mathrm{Pic}(A)\simeq \lim_n \mathrm{Pic}(A/I_n)$, and we let $A\{1\}\in \mathrm{Pic}(A)$ correspond to $\left((\phi^n)^*I\otimes_{A}A/I_n\right){n\ge 0}$. See \cite[\secsymb4.9]{Drin} for additional details, or \cite[\secsymb2]{BL} for a more explicit construction bootstrapping from the case where $A/I$ is $\pi$-torsion free.

Taking $(A,I) = (\S_L, (q_n(\omega)))$ once more, we give also a more explicit definition. We can define $A\{1\}$ by
\[\plim\left(
		\begin{tikzcd}
			\cdots\ar[twoheadrightarrow]{r}{1/\pi} &I_3/I_3^2\ar[twoheadrightarrow]{r}{1/\pi} &I_2/I_2^2\ar[twoheadrightarrow]{r}{1/\pi} &I_1/I_1^2
		\end{tikzcd}
		\right).\]
Here 
\[I_m = \left(q_n(\omega)q_{n + 1}(\omega)\cdots q_{n + m - 1}(\omega)\right) = \left(\frac{[\pi^{n + m - 1}](\omega)}{[\pi^{n - 1}](\omega)}\right)\]
and the transition maps $I_{m+1}/I_{m+1}^2\rightarrow I_m/I_m^2$ are quotienting by $I_m^2$ followed with dividing by $\pi$; these are well-defined and surjective as
\[\frac{[\pi^{n + m}](\omega)}{[\pi^{n - 1}](\omega)} \equiv \pi\frac{[\pi^{n + m - 1}](\omega)}{[\pi^{n - 1}](\omega)}\mod{\left(\frac{[\pi^{n + m - 1}](\omega)}{[\pi^{n - 1}](\omega)}\right)^2}\]
since $[\pi^{n + m}](\omega) = [\pi]\left([\pi^{n + m - 1}](\omega)\right)$ and $[\pi](T) = \pi T + (\text{higher order terms})$.

\begin{lem}~\label{lem:qi-lcm}
	\begin{enumerate}
		\item[(1)] $A/I_m$ is $\pi$-torsionfree for all $m\ge 1$.
		\item[(2)] We have $I_m = \bigcap_{i = 0}^{m - 1} \phi^i(I) = \bigcap_{i = 0}^m \left(q_{n + i}(\omega)\right)$.
	\end{enumerate}
\end{lem}
\begin{proof}
	We prove part (1) by induction. The result is clear for $m = 1$, and for $m\ge 2$ we have an exact sequence
	\[0\longrightarrow I_m\otimes_A A/\phi^m(I) \cong I_m/I_{m + 1}\longrightarrow A/I_{m + 1}\longrightarrow A/I_m\longrightarrow 0\]
	where the first and third terms are $\pi$-torsion free.

	For part (2) we follow \cite[lemmas~2.2.8, 2.2.9]{BL}. First, we show that the natural map $f:\phi^m(I)/I_{m + 1}\rightarrow A/I_m$ is injective. As above, using the identity $[\pi](T) = \pi T + (\text{higher order terms})$ one shows that the $f$ has image containing $(\pi, I_m)$; by part (1), $f$ therefore factors as $f = \pi f_0$. We show that $f_0$ is an isomorphism; as the domain and codomain are invertible $A/I_m$ modules, it suffices to show surjectivity. One shows by induction over $m\ge 1$ that if $\alpha\in I$ then $f_0(\phi^m(\alpha)) \mod{(\pi,I)}$ is a unit in $A/(\pi,I)$. Then by $(\pi,I)$-adic completeness and the inclusion $I_m\subseteq (\pi,I)$ we conclude that the image of $I\stackrel{\phi^m}{\rightarrow} \phi^m(I)/I_m\stackrel{f_0}{\rightarrow} A/I_m$ is the unit ideal as desired.

	We now prove the statement in the lemma by induction on $m\ge 0$, with $m = 0$ being interpreted as the equality $(1) = (1)$ of unit ideals. For $m\ge 1$, let $\alpha\in \bigcap_{i = 0}^m\phi^i(I)$. By induction, we have $\alpha\in I_m\cap \phi^m(I)$. Thus $\alpha$ is in the kernel of $\phi^m(I)\rightarrow A/I_m$, which factors as
	\[\phi^m(I)\rightarrow \phi^m(I)/I_{m + 1}\stackrel{f}{\rightarrow} A/I_m.\]
	Since we showed that $f$ is injective, we have that $\alpha\in I_{m + 1}$ as desired.
\end{proof}
We now define $\log_\Prism$. Let $I_{\phi = [\pi]}$ denote the subset of $\alpha\in I$ such that $\phi(\alpha) = [\pi](\alpha)$. For example, we have that $[\pi^n](\omega)\in I_{\phi = [\pi]}$ since $\phi([\pi^n](\omega)) = [\pi^n]([\pi](\omega)) = [\pi]([\pi^n](\omega))$. 
\begin{lem}\label{lem:pi-alpha-in-I2}
 	If $\alpha\in I_{\phi = [\pi]}$ and $m\ge 1$ then $[\pi^m](\alpha)\in I_{m + 1}$ and $[\pi^m](\alpha)\equiv \pi \cdot [\pi^{m - 1}](\alpha)\pmod{I_m^2}$. 
 \end{lem}
 \begin{proof}
 	The second part of the lemma is clear from $[\pi](T) = \pi T + (\text{higher order terms})$. For the first part, for each $0\le i \le m$ we have $[\pi^m](\alpha) = [\pi^{m - i}]([\pi^i](\alpha))\in \phi^i(I)$. Thus $[\pi^m](\alpha)\in I_{m + 1}$ by lemma~\ref{lem:qi-lcm}(2).
 \end{proof}
\begin{defn}
	Let $\log_\Prism:I_{\phi = [\pi]}\rightarrow \S_L\{1\}$ be defined by
	\[\log_\Prism(\alpha) = ([\pi^{m - 1}](\alpha))_{m\ge 1} = (\phi^{m-1}(\alpha))_{m\ge 1}\in \plim_{1/\pi} I_m/I_m^2 = \S_L\{1\}.\]
\end{defn}

Recall that the Lubin-Tate formal $\O_L$-module $\G$ comes with a formal group law $X+_\G Y\in \O_L\ps{X,Y}$ satisfying
\begin{align}
	X +_\G Y &= X + Y + (\text{degree }\ge 2\text{ terms}) \\
	[a](X +_\G Y) &= [a](X) +_\G [a](Y)\quad\quad\text{for }a\in \O_L.
\end{align}
This second condition with $a = \pi$ implies that if $\alpha,\beta\in I_{\phi = [\pi]}$ then $\alpha+_\G\beta\in I_{\phi = [\pi]}$ as well. Similarly, we have that if $\alpha\in I_{\phi=[\pi]}$ and $a\in \O_L$ then $[a](\alpha)\in I_{\phi = [\pi]}$. Thus $I_{\phi = [\pi]}$ can be viewed as an $\O_L$-module. We show that $\log_\Prism$ is an $\O_L$-module homomorphism.
\begin{prop}
	For $\alpha,\beta\in I_{\phi = [\pi]}$ and $a\in \O_L$ we have $\log_\Prism(\alpha +_\G \beta) = \log_\Prism(\alpha) + \log_\Prism(\beta)$ and $\log_\Prism([a](\alpha)) = a\log_\Prism(\alpha)$.
\end{prop}
\begin{proof}
	We have 
	\begin{align*}
		\log_\Prism(\alpha +_\G\beta) &= ([\pi^{m - 1}](\alpha +_\G \beta))_{m\ge 1} = ([\pi^{m - 1}](\alpha) +_\G [\pi^{m - 1}](\beta))_{m\ge 1} \\
		&= ([\pi^{m - 1}](\alpha) + [\pi^{m - 1}](\beta))_{m\ge 1} = \log_\Prism(\alpha) + \log_\Prism(\beta)
	\end{align*}
	where the penultimate equality uses that $X +_\G Y = X + Y + (\text{degree }\ge 2\text{ terms})$ and $[\pi^{m-1}](\alpha),[\pi^{m - 1}](\beta)\in I_m$ by lemma~\ref{lem:pi-alpha-in-I2}. The identity $\log_\Prism([a](\alpha)) = a\log_\Prism(\alpha)$ is shown similarly.
\end{proof}
\begin{rmk}
	Recall that $\S_L$ was defined by applying construction~\ref{const:algebraic-construction} to an element $e =\in T\G$ of the Tate module of $\G$; this gave a map $\iota:\O_L\ps{T}\rightarrow W_L(\O_{L_\infty}^\flat)$ with image $\S_L$ and the element $\omega := \iota(T)$. As in remark~\ref{rmk:change-by-unit}, applying the same construction with the element $e' = ae$ for some $a\in \O_L$ results in the element $\omega' = [a](\omega)$ still in $\S_L$. We thus get a natural $\O_L$-module map
	\begin{align*}
		\rho:&T\G\rightarrow I_{\phi = [\pi]} \\
		ae&\mapsto [a\pi^n](\omega)
	\end{align*}
	and by composition an $\O_L$-module map $T\G\rightarrow \S_L\{1\}$. 
\end{rmk}

\section{Laurent \texorpdfstring{$F$}{F}-crystals}\label{sec:F-crystals}

In this section we introduce \'etale $\varphi_q$-modules over $L$-typical prisms and Laurent $F$-crystals, and we prove theorem~\ref{thm:intro-version-main-result}. In \secsymb\ref{ssec:phi-modules-over-prisms} we show that the equivalence $\Mod_{\A_K}^{\varphi_q,et}\simeq \Mod_{W_L(K_\infty^\flat)}^{\varphi_q,et}$ is in fact a special case of an equivalence
\[\Mod_{(A,I)}^{\varphi_q,et}\simeq \Mod_{(A,I)_{\perf}}^{\varphi_q,et}\]
between categories of \'etale $\varphi_q$-modules which make sense for any $L$-typical prism $(A,I)$. In \secsymb\ref{ssec:laurent-F-crystals}, we define Laurent $F$-crystals in the $L$-typical prismatic setting; these are objects which serve as relativizations of \'etale $\varphi_q$-modules over a base formal scheme $X/\O_L$. We go on to show that the category of Laurent $F$-crystals over $X$ is equivalent to the category of lisse local systems on the adic generic fiber $X_\eta$ with coeffficients in $\O_L$. Finally, in \secsymb\ref{ssec:phi-G-modules-F-crystals} we use this theory to recover the Kisin-Ren equivalence between Lubin-Tate $(\varphi_q,\Gamma)$-modules and continuous $G_K$ representations over $\O_L$.

\subsection{\'Etale \texorpdfstring{$\varphi_q$}{phi\_q}-modules over \texorpdfstring{$L$}{L}-typical prisms}\label{ssec:phi-modules-over-prisms}

Given a $p$-adic field $K/L$ and a Lubin-Tate formal $\O_L$-module, we described in \secsymb\ref{sec:LT-phi-G-modules} prisms $(\A_K^+, (q_n(\omega))$ with perfection $(\Ainf(\O_{K_\infty}), \ker\theta)$. We also saw that the categories of \'etale $\varphi_q$-modules over $\A_K = \A_K^+[1/q_n(\omega)]^\wedge_{(\pi)}$ and $W_L(K_\infty^\flat) = \Ainf(\O_{K_\infty})[1/\ker\theta]^\wedge_{(\pi)}$ were equivalent. In fact, this reflects a general fact about categories of $\varphi_q$-modules over $L$-typical prisms, which we prove here.

The definition of $\varphi_q$-modules in this setting is as follows.
\begin{defn}~
	\begin{enumerate}
		\item[(1)] Let $\mc A$ be a ring together with a ring homomorphism $\varphi:\mc A\rightarrow \mc A$. An \'etale $\varphi$-module over $\mc A$ is a finite projective $\mc A$-module $M$ equipped with an isomorphism 
		\[\varphi_M:\varphi^*M := \mc A \otimes_{\varphi,\mc A}M\stackrel{\sim}{\longrightarrow}M.\]
		This gives us a $\varphi$-semilinear map $M\rightarrow M$ via $m\mapsto \varphi_M(1\otimes m)$; we will abuse notation and write $\varphi_M$ also for this map. Equipped with $\mc A$-module endomorphisms commuting with the $\varphi_M$'s, \'etale $\varphi$--modules over $\mc A$ form a category $\Mod_{\mc A}^{\varphi, et}$. 
		\item[(2)] Let $(A,I)$ be a bounded $L$-typical prism. Then an \'etale $\varphi_q$-module over $(A,I)$ is an \'etale $\varphi = \phi_{\mc A}$-module over $\mc A = A[\frac1I]^\wedge_{(\pi)}$ in the sense of (1). In other words, it is a finite projective $\mc A$-module $M$ with an isomorphism $\varphi_M:\varphi^*M\stackrel{\sim}{\rightarrow}M$ (which we also view as a $\varphi$-semilinear endomorphism of $M$). We denote the resulting category by $\Mod_{(A,I)}^{\varphi_q,et} = \Mod_{\mc A}^{\phi_{\mc A},et}$.
		\item[(3)] For the corresponding category of derived objects, let $D_{\perf}(\mc A)$ denote the category of perfect complexes in modules over the ring $\mc A$, i.e. objects in the derived category of $\mc A$-modules quasi-isomorphic to a bounded complex of finite projective $\mc A$-modules. If $\mc A$ has an endomorphism $\varphi$ then we write $D_{\perf}(\mc A)^{\varphi = 1}$ for the category of pairs $(E,\varphi_E)$ where $E\in D_{\perf}(\mc A)$ and $\varphi_E:\varphi^*E\stackrel{\sim}{\rightarrow}E$.
	\end{enumerate}
\end{defn}

On the representation-theory side, the appropriate generalization of $G_{K_\infty}$-representations on finite free $\Z_p$-modules is $\O_L$-local systems on $\Spec(\mc A/\pi)$. Recall that this means the following.
\begin{defn} (c.f. \cite[definition~8.1]{S}.)
	Let $X$ be a scheme, formal scheme, or adic space, and denote by $X_{et}$ the \'etale site of $X$.
	\begin{enumerate}
		\item[(1)] For $n\ge 1$, an $\O_L/\pi^n$-local system on $X_{et}$ is a sheaf of flat $\O_L/\pi^n$-modules on $X_{et}$ which is locally a constant sheaf associated to a finitely generated $\O_L/\pi^n$-module. We denote this category by $\Loc_{\O_L/\pi^n}(X)$.
		\item[(2)] An $\O_L$-local system on $X_{et}$ is an inverse system $(\mathbb{L}_n)_{n\ge 1}$ of $\O_L/\pi^n$-local systems on $X_{et}$ in which the transition maps induce isomorphisms $\mathbb{L}_{n + 1}/\pi^n\stackrel{\sim}{\rightarrow}\mathbb{L}_n$. We denote this category by $\Loc_{\O_L}(X)$. This identifies with the category of lisse $\hat{\O}_L$-sheaves on the pro-\'etale site $X_{proet}$.
		\item[(3)] Let $D_{lisse}^b(X,\O_L)$ be the subcategory of the derived category of $\hat\O_L$-modules on $X_{proet}$ spanned by objects $T$ which are locally bounded, derived $\pi$-complete, and have $H^i(X_{proet}, T/\pi)$ locally constant with finitely generated stalks. 
	\end{enumerate}
	When $X = \Spec R$ is affine, we simplify notation by writing $\Loc_{\O_L}(R)$ for $\Loc_{\O_L}(\Spec R)$ and similarly for $D_{lisse}^b$. 
\end{defn}
\begin{rmk}
	For a field $K$ we have equivalences $\Loc_{\O_L/\pi^n}(K)\cong \Rep_{\O_L/\pi^n}(G_K)$ and $\Loc_{\O_L}(K) \cong \Rep_{\O_L}(G_K)$ with the categories of continuous $G_K$-representations on finite free $\O_L/\pi^n$- or $\O_L$-modules
\end{rmk}

The main result of this section is as follows.

\begin{prop}\label{prop:perfection-base-change}
	Let $(A,I)$ be a bounded $L$-typical prism. Let $(A_{\perf}, IA_{\perf})$ be the perfection of $(A,I)$ as in proposition~\ref{prop:prism-perfection}. Then base change gives an equivalence
	\begin{align*}
		\Mod_{(A,I)}^{\varphi_q,et}&\longrightarrow \Mod_{(A_{\perf}, IA_{\perf})}^{\varphi_q,et} \\
		M&\mapsto A_{\perf}[\tfrac1I]^\wedge\otimes_{A[\frac1I]^\wedge} M.
	\end{align*}
	Both of these categories are in turn equivalent to $\Loc_{\O_L}(A[\tfrac1I]/\pi)$. We similarly have equivalences 
	\[D_{\perf}(A[\tfrac{1}{I}]^\wedge_{(\pi)})^{\phi = 1} \simeq D_{\perf}(A_{\perf}[\tfrac{1}{I}]^\wedge_{(\pi)})^{\phi = 1}\simeq D_{lisse}^b(A[\tfrac{1}{I}]/\pi,\O_L).\]
\end{prop}
\begin{rmk}\label{rmk:functor-description}
	If $(A,I) = (W_L(R^\flat), \ker\theta)$ is a perfect $L$-typical prism, then the equivalence of the theorem is given by
	\begin{align*}
		\Mod_{(A,I)}^{\varphi_q,et}&\simeq \Loc_{\O_L}(R[\tfrac{1}{\pi}]) \\ 
		M&\mapsto \left(R[\tfrac{1}{\pi}]_{et}\ni S \mapsto \left(W_L(S^\flat)\otimes_{W_L(R[\tfrac1\pi]^\flat)} M/\pi^n\right)^{\phi = 1}\right)_{n\ge 1}.
	\end{align*}
	where $(-)^{\phi = 1}$ denotes taking fixed points for $\phi = \phi_{W_L(S^\flat)}\otimes \phi_M$. The same formula holds for the derived categories, with the tensor replaced by $\otimes^L$ and with the inverse system replaced with $R\lim$ of the inverse system.
\end{rmk}
\begin{rmk}
	Theorem~\ref{thm:kisin-ren-equiv} follows from proposition~\ref{prop:perfection-base-change}: taking $(A,I) = (\A_K^+,(q_n(\omega)))$, we get
	\[\Mod_{\A_K}^{\varphi_q,et}\simeq \Mod_{W_L(K_\infty^\flat)}^{\varphi_q,et} \simeq \Loc_{\O_L}(K_\infty^\flat)\simeq \Rep_{\O_L}(G_{K_\infty}).\]
	We will discuss this point further in \secsymb~\ref{ssec:phi-G-modules-F-crystals}.
\end{rmk}

The key input to the proof of proposition~\ref{prop:perfection-base-change} is the following comparison between $\pi$-torsion $\varphi_q$-modules and $\F_q$-local systems. 
\begin{lem}\label{lem:lisse-sheaves-torsion-phi-modules}
	Let $R$ be a $\F_q$-algebra. Then there is an equivalence of categories
	\begin{align*}
		\Mod_{R}^{\varphi_q,et}&\simeq\Loc_{\F_q}(R) \\
		M&\mapsto (R_{et}\ni S\mapsto S\otimes_R M)^{\varphi_q = 1} \\
		(\O_{R,et}\otimes_{\F_q}T)(R) &\mapsfrom T.
	\end{align*}
	The corresponding derived statement $D_{\perf}(R)^{\varphi_q = 1}\simeq D_{lisse}^b(R, \F_q)$ also holds.
\end{lem}
\begin{proof}
	Using the same argument in \cite[proposition~3.6]{BScrys}, we reduced the derived statement to the statement $\Mod_R^{\varphi_q,et}\simeq \Loc_{\F_q}(R)$. But this is well known and due originally to Katz \cite[proposition~4.1.1]{K}.
\end{proof}

\begin{proof}[Proof of proposition~\ref{prop:perfection-base-change}]
	We explain the proof for $\Mod_{(A,I)}^{\varphi_q,et}$, with the derived version being identical. First, we show that the base change functor is an equivalence. By the $\pi$-adic completeness of $A[\tfrac1I]^\wedge_{(\pi)}$ and d\'evissage, we reduce to the $\pi$-torsion case, i.e. to showing that base change gives an equivalence
	\[\Mod_{A[\tfrac1I]/\pi}^{\varphi_q,et}\stackrel{\sim}{\longrightarrow}\Mod_{A_{\perf}[\tfrac1I]/\pi}^{\varphi_q,et}.\]
	Applying lemma~\ref{lem:lisse-sheaves-torsion-phi-modules} with $R = A[\tfrac1I]/\pi$, we are reduced to showing that base change gives an equivalence
	\[\Loc_{\F_q}(A/\pi[\tfrac1I])\simeq \Loc_{\F_q}(A_{\perf}/\pi[\tfrac1I]).\]
	As $I$ is a Cartier divisor, we may assume that $I$ is generated by a nonzerodivisor $d\in A$. Then this equivalence holds since the maps
	\[A/\pi[\tfrac{1}{d}]\longrightarrow (\varinjlim_{\varphi_q}A/\pi)[\tfrac{1}{d}]\longrightarrow (\varinjlim_{\varphi_q} A/\pi)^\wedge_{(d)}[\tfrac{1}{d}]\] 
	induce equivalences of \'etale sites (the first by topological invariance of the \'etale site and the second by \cite[proposition~5.4.53]{GR}).

	For the identification with $\Loc_{\O_L}(A[\tfrac{1}{I}])$, note that as $A_{\perf}[\tfrac{1}{I}]^\wedge_{(\pi)}$ is a $\pi$-adically complete perfect $\delta_L$-algebra, we have $A_{\perf}[\tfrac{1}{I}]^\wedge = W_L(A_{\perf}[\tfrac{1}{I}]/\pi)$ by proposition~\ref{prop:perfect-d-algs-equiv}. Thus by $\pi$-adic completeness and lemma~\ref{lem:lisse-sheaves-torsion-phi-modules} we get
	\[\Mod_{(A_{\perf}, I)}^{\varphi_q,et}\simeq \Loc_{\O_L}(A_{\perf}[\tfrac{1}{I}]/\pi)\]
	which identifies in turn with $\Loc_{\O_L}(A[\tfrac{1}{I}]/\pi)$ by the same argument as above.
\end{proof}

As a corollary of proposition~\ref{prop:perfection-base-change}, we get that the equivalence $D_{\perf}(A[\tfrac{1}{I}]^\wedge_{(\pi)})^{\phi = 1}\simeq D_{\perf}(A_{\perf}[\tfrac1I]^\wedge_{(\pi)})^{\phi = 1}$ also holds ``on the level of objects.''
\begin{cor}\label{cor:perfection-base-change-phi-invariants}
	Let $(A,I)$ be a bounded $L$-typical prism, and let $M\in D_{\perf}(A[\tfrac1I]^\wedge_{(\pi)})^{\phi = 1}$. Then the canonical map
	\[M^{\phi = 1}\longrightarrow (A_{\perf}[\tfrac1I]^\wedge_{(\pi)}\otimes_{A[\tfrac1I]^\wedge_{(\pi)}}M)^{\phi = 1}\]
	is an isomorphism.
\end{cor}
\begin{proof}
	Our proof will follow \cite[lemma~6.3]{Guo}. First we recall how $M^{\phi = 1}$ is defined. In general, let $B$ be a ring with an endomorphism $\varphi$ and let $B[F]$ be the noncommutative polynomial ring with relation $Fb = \varphi(b)F$. Then we get a fully faithful embedding $D_{\perf}(B)^{\varphi = 1}\hookrightarrow D(B[F])$ into the derived category by sending $(N,\varphi_N:\varphi^*N\stackrel{\sim}{\rightarrow}N)\in D_{\perf}(B)^{\varphi = 1}$ to the $B$-algebra $N$ with $F$-action given by $N\rightarrow (\varphi_N)_*N$ (this is the normal way of seeing an element of $D_{\perf}(B)^{\varphi = 1}$ as being a $B$-module with a $\varphi$-semilinear endomorphism). Then $N^{\varphi = 1}$ is defined by
	\[N^{\varphi = 1} := R\mathrm{Hom}(B[F]/(1 - F)B[F], N).\]

	Thus, setting $\mc A = A[\tfrac1I]^\wedge_{(\pi)}$ and $\mc A_{\perf} = A_{\perf}[\tfrac1I]^\wedge_{(\pi)}$ to simplify notation, our goal is to show that
	\[R\mathrm{Hom}(\mc A[F]/(1 - F), M)\longrightarrow R\mathrm{Hom}(\mc A_{\perf}[F]/(1 - F), \mc A_{\perf}\otimes_{\mc A} M)\]
	is an isomorphism. As this can be checked on cohomology and $D_{\perf}(\mc A)$ is closed under shifting, it thus suffices to show that
	\[\mathrm{Hom}(\mc A[F]/(1 - F), M)\longrightarrow \mathrm{Hom}(\mc A_{\perf}[F]/(1 - F), \mc A_{\perf}\otimes_{\mc A} M)\]
	is an isomorphism. But, up to fully faithful embedding, the hom-set on the right comes from the one of the left by applying the functor $M\mapsto \mc A_{\perf}\otimes_{\mc A} M$, which is an equivalence by proposition~\ref{prop:perfection-base-change}. Thus the hom-sets are isomorphic as desired.
\end{proof}

\subsection{Laurent \texorpdfstring{$F$}{F}-crystals}\label{ssec:laurent-F-crystals}
For a bounded formal scheme $X$ adic over $\Spf \O_L$, denote by $\O_\Prism$ the presheaf $(A,I)\mapsto A$ on the $L$-typical prismatic site $X_{\Prism_L}$. By $(\pi, I)$-completely faithfully flat descent (see \cite[corollary~3.12]{BS}), $\O_\Prism$ is a sheaf, which we take as the structure sheaf for $X_{\Prism_L}$. It has a natural endomorphism $\phi$ lifting $\varphi_q$ on $\O_\Prism/\pi$ and an ideal sheaf $\mc I\subseteq \O_{\Prism}$ given by $(A,I)\mapsto I$. We will also make use of the sheaf $\O_{\Prism,\perf}$, which sends $(A,I)\mapsto A_{\perf}$.

Denote by $\O_\Prism[\tfrac{1}{\mc I}]^\wedge_{(\pi)}$ the $\pi$-adic completion of the localization of $\O_\Prism$ away from $\mc I$ (i.e. locally inverting a generator of $\mc I$; recall that if $(A,I)$ is a prism then $I$ is a Cartier divisor hence locally principal). 

\begin{defn} Let $X$ be a bounded formal scheme adic over $\Spf \O_L$.
	\begin{enumerate}
		\item[(1)] A Laurant $F$-crystal is a finite locally free $\O_{\Prism}[\tfrac{1}{\mc I}]^\wedge_{(\pi)}$-module $\mc M$ over $X_{\Prism_L}$ equipped with an isomorphism
		\[F:\phi^*\mc M\stackrel{\sim}{\longrightarrow} \mc M.\]
		As before, we abusively write $\phi_{\mc M}:\mc M\rightarrow \mc M$ also for the resulting $\phi$-semilinear endomorphism. 
		\item[(2)] Write $\Vect(\mathcal X, \O)$ for the category of vector bundles on a ringed topos $(\mathcal X, \O)$, we can describe the category of Laurent $F$-crystals over $X_{\Prism_L}$ as $\Vect(X_{\Prism_L}, \O_\Prism[\tfrac{1}{\mc I}]^\wedge_{\pi})^{\phi=1}$, the $\phi$-fixed objects of $\Vect(X_{\Prism_L}, \O_\Prism[\tfrac{1}{\mc I}]^\wedge_{(\pi)})$.
		\item[(3)] Similarly, write $D_{\perf}(\mc X, \O)$ for the category of perfect complexes on $(\mc X, \O)$, i.e. objects $E$ in the derived category of $\O$-modules over $\mc X$ such that there is a cover $\{U_i\}$ of $\mc X$ with each $E|_{U_i}$ a perfect complex of $\O(U_i)$-modules. Let $D_{\perf}(\mc X, \O)^{\phi = 1}$ denote corresponding category of $\phi$-fixed objects.
	\end{enumerate}
\end{defn}

Given a Laurent $F$-crystal $\mc M$ and an object $(A,I)\in X_{\Prism_L}$, we have that $\mc M(A,I)\in \Mod_{(A,I)}^{\varphi_q,et}$ is an \'etale $\varphi_q$-module. We further have the following.
\begin{lem}\label{lem:comp-ff-descent-vect}
	There is an equivalence
	\begin{align*}
		\Vect(X_{\Prism_L}, \O_{\Prism}[\tfrac{1}{\mc I}]^\wedge_{(\pi)})^{\phi = 1}&\stackrel{\sim}{\longrightarrow} \lim_{(A,I)\in X_{\Prism_L}}\Mod_{(A,I)}^{\varphi_q,et} \\
		\mc M&\mapsto (\mc M(A,I))_{(A,I)\in X_{\Prism_L}}.
	\end{align*}
	Similarly $D_{\perf}(X_{\Prism_L}, \O_{\Prism}[\tfrac{1}{\mc I}]^\wedge_{(\pi)})^{\phi = 1} \simeq \lim_{(A,I)\in X_{\Prism_L}} D_{\perf}(A[\tfrac{1}{I}]^\wedge_{(\pi)})^{\phi = 1}$. A similar result holds with $\O_{\Prism,\perf}$ replacing $\O_{\Prism}$ (and $\Mod_{(A,I)_{\perf}}^{\varphi_q,et}$ replacing $\Mod_{(A,I)}^{\varphi_q,et}$).
\end{lem}
\begin{proof}
	The proof is the same as \cite[proposition~2.7]{BScrys}: one can reduce via devissage to the $\pi$-torsion case, where the result follows from the descent results in \cite[theorem~5.8]{Matthew-desc}.
\end{proof}

We regard Laurent $F$-crystals as (geometrically) relativizing \'etale $\varphi_q$-modules over the base formal scheme $X$. We then have the following analogues of proposition~\ref{prop:perfection-base-change} and corollary~\ref{cor:perfection-base-change-phi-invariants} (except without the local systems, which will appear in theorem~\ref{thm:F-crystals-local-systems}).

\begin{thm}\label{thm:F-crystal-base-change}
	Let $X$ be a bounded formal scheme adic over $\Spf\O_L$.
	\begin{enumerate}
		\item[(1)] Base change induces an equivalence of categories
		\begin{align*}
			\Vect(X_{\Prism_L},\O_\Prism[\tfrac{1}{\mc I}]^\wedge_{(\pi)})^{\phi = 1}&\stackrel{\sim}{\longrightarrow}\Vect(X_{\Prism_L},\O_{\Prism,\perf}[\tfrac{1}{\mc I}]^\wedge_{(\pi)})^{\phi = 1} \\
			\mc M&\mapsto \O_{\Prism,\perf}[\tfrac{1}{\mc I}]^\wedge_{(\pi)}\otimes_{\O_\Prism[\tfrac{1}{\mc I}]^\wedge_{(\pi)}}\mc M
		\end{align*}
		and the same holds with $D_{\perf}$ replacing $\Vect$.
		\item[(2)] For $\mc M\in D_{\perf}(X_{\Prism_L}, \O_\Prism[\tfrac{1}{\mc I}]^\wedge_{(\pi)})^{\phi = 1}$, the canonical map
		\[\mc M^{\phi = 1}\longrightarrow  (\O_{\Prism,\perf}[\tfrac{1}{\mc I}]^\wedge_{(\pi)}\otimes_{\O_\Prism[\tfrac{1}{\mc I}]^\wedge_{(\pi)}}\mc M)^{\phi = 1}\]
		is an isomorphism.
	\end{enumerate}
\end{thm}
\begin{proof}
	For part (1), we have the following commutative diagram. 
	\begin{center}
		\begin{tikzcd}
			\Vect(X_{\Prism_L},\O_\Prism[\tfrac{1}{\mc I}]^\wedge_{(\pi)})^{\phi = 1}\ar{ddd}\ar[dashed]{rrr} &[-80pt] & &[-90pt]\Vect(X_{\Prism_L},\O_{\Prism,\perf}[\tfrac{1}{\mc I}]^\wedge_{(\pi)})^{\phi = 1}\ar{ddd} \\
				& \mc M\ar[mapsto]{d}\ar[dashed,mapsto]{r} & \O_{\Prism,\perf}[\tfrac{1}{\mc{I}}]^\wedge_{(\pi)}\otimes_{\O_{\Prism}[\tfrac{1}{\mc I}]^\wedge_{(\pi)}} \mc M\ar[mapsto]{d}\\
				& (\mc M(A,I))_{(A,I)\in X_{\Prism_L}}\ar[mapsto]{r} & \left(A_{\perf}[\tfrac1I]^\wedge_{(\pi)}\otimes_{A[\tfrac1I]^\wedge_{(\pi)}}\mc M(A,I)\right)_{(A,I)\in X_{\Prism_L}}\\
			\displaystyle\lim_{(A,I)\in X_{\Prism_L}} \Mod_{(A,I)}^{\varphi_q,et}\ar{rrr} &&&\displaystyle\lim_{(A,I)\in X_{\Prism_L}^{\perf}} \Mod_{(A,I)_{\perf}}^{\varphi_q,et}
		\end{tikzcd}
	\end{center}
	By lemma~\ref{lem:comp-ff-descent-vect} the vertical arrows are equivalences of categories, and the bottom horizontal arrow is an equivalence by proposition~\ref{prop:perfection-base-change}. The same holds replacing $\Vect$ with $D_{\perf}$ and $\Mod^{\varphi_q,et}_{(A,I)}$ with $D_{perf}(A[\tfrac1I]^\wedge_{(\pi)})$. For part (2), we can again check on individual prisms $(A,I)\in X_{\Prism_L}$, in which case the result follows from corollary~\ref{cor:perfection-base-change-phi-invariants}.
\end{proof}
Let $X_{\Prism_L}^{\perf}$ denote the subsite of $X_{\Prism_L}$ consisting of perfect $L$-typical prisms.
\begin{cor}\label{cor:F-crystal-perfect-site}
	For $X$ a bounded formal scheme adic over $\Spf\O_L$ we have
	\[\Vect(X_{\Prism_L},\O_\Prism[\tfrac1{\mc I}]^\wedge_{(\pi)})^{\phi = 1} \simeq \lim_{(A,I)\in X_{\Prism_L}^{\perf}}\Mod_{(A,I)}^{\varphi_q,et}\]
	and similarly for $D_{\perf}$. 
\end{cor}
\begin{proof}
	This follows from theorem~\ref{thm:F-crystal-base-change}(1), lemma~\ref{lem:comp-ff-descent-vect}, and the fact that for $\mc M\in D_{\perf}(X_{\Prism_L},\O_\Prism[\tfrac1{\mc I}]^\wedge_{(\pi)})^{\phi = 1}$ and $(A,I)\in X_{\Prism_L}$ we have
	\[\mc M((A,I)_{\perf}) \cong A_{\perf}[\tfrac1I]^\wedge_{(\pi)}\otimes_{A[\tfrac1I]^\wedge_{(\pi)}}\mc M(A,I).\]
\end{proof}

We now globalize the relationship between \'etale $\varphi_q$-modules and local systems from proposition~\ref{prop:perfection-base-change}. We've essentially already shown this in the case that $X = \Spf(R)$ for a perfectoid $\O_L$-algebra:

\begin{prop}\label{prop:F-crystal-over-perfectoid-local-system}
	If $R$ is a perfectoid $\O_L$-algebra, there are equivalences
		\[\Vect(R_{\Prism_L}, \O_\Prism[\tfrac{1}{\mc I}]^\wedge_{(\pi)})^{\phi=1} \simeq \Mod_{(\Ainf(R), \ker\theta)}^{\varphi_q,et}\simeq \Loc_{\O_L}(R[\tfrac1\pi])\]
		and $D_{\perf}(R_{\Prism_L}, \O_\Prism[\tfrac{1}{\mc I}]^\wedge_{(\pi)})^{\phi=1}\simeq D_{\perf}(W_L(R[\tfrac1\pi]^\flat))^{\phi = 1}\simeq D_{lisse}^b(R[\tfrac1\pi], \O_L)$.
\end{prop}
\begin{proof}
	By proposition~\ref{prop:prism-perfection}(2), $R_{\Prism_L}^{\perf}$ has an initial object $(\Ainf(R),\ker\theta)$. By corollary~\ref{cor:F-crystal-perfect-site} and proposition~\ref{prop:perfection-base-change}, we then have that 
	\[\Vect(R_{\Prism_L}, \O_\Prism[\tfrac{1}{\mc I}^\wedge_{(\pi)})^{\phi = 1}\simeq \Loc_{\O_L}(\Ainf(R)\tfrac{1}{\ker\theta}]/\pi).\]
	In the proof of proposition~\ref{prop:perfectoid-charactierzation}, we showed that $\ker\theta$ has a generator of the form $d = [a_0] - \pi u$ for $a_0\in R^\flat$ such that $R^\flat$ is $a_0$-adically complete and $a_0^\sharp = \pi$. Thus $\Ainf[\tfrac{1}{\ker\theta}]/\pi\cong R^\flat[\tfrac{1}{a_0}]$, and we conclude by the tilting equivalence.
\end{proof}
\begin{cor}\label{cor:etale-comparison-perfectoid}
	If $R$ is perfectoid and $\mc M\in D_{\perf}(R_{\Prism_L}, \O_\Prism[\tfrac{1}{\mc I}]^\wedge_{(\pi)})^{\phi=1}$ corresponds to $T\in D_{lisse}^b(R[\tfrac1\pi]_{et}, \O_L)$ under the equivalence of proposition~\ref{prop:F-crystal-over-perfectoid-local-system}, then there is an isomorphism
	\[R\Gamma(R_{\Prism_L}, \mc M)^{\phi = 1}\cong R\Gamma(R[\tfrac1\pi]_{proet}, T).\]
\end{cor}
\begin{proof}
	Since the map $D_{\perf}(R_{\Prism_L}, \O_\Prism[\tfrac{1}{\mc I}]^\wedge_{(\pi)})^{\phi=1}\rightarrow D_{\perf}(W_L(R[\tfrac1\pi]^\flat))^{\phi = 1}$ is given by $\mc M\mapsto R\Gamma(R_{\Prism_L}, \mc M)$, this follows from the description of the map $D_{\perf}(W_L(R[\tfrac1\pi]^\flat))^{\phi = 1}\rightarrow D_{lisse}^b(R[\tfrac1\pi], \O_L)$ given in remark~\ref{rmk:functor-description}.
\end{proof}

The following theorem globalizes proposition~\ref{prop:F-crystal-over-perfectoid-local-system} by descent from the affine perfectoid case. This generalizes \cite[cor~3.8]{BScrys}. More specifically, we will use v-descent: by \cite[\secsymb15]{SchDiamonds}, $X_\eta$ can be viewed as a locally spatial diamond, so that the categories $\Loc_{\O_L}(X_\eta)$ and $D_{lisse}^b(X_\eta,\O_L)$ satisfy v-descent with respect to v-covers of $X_\eta$ (i.e. covers by surjective maps of v-sheaves; see \cite{MW} and especially \cite[theorem~3.11]{MW} for the relationship between local systems on the diamondification of $X_\eta$ and pro-\'etale local systems on $X_\eta$). By \cite[lemma~15.3]{SchDiamonds}, any analytic adic space has a v-cover by generic fibers of perfectoid rings; by theorem~\ref{thm:perfectoid-OL-alg-perfectoid-ring} this also gives a v-cover by perfectoid $\O_L$-algebras.

For now this globalization will result in losing the \'etale $\varphi_q$-module part of the result; that part will be restored in the special case $X = \Spf\O_K$ in \secsymb\ref{ssec:phi-G-modules-F-crystals}.

\begin{thm}\label{thm:F-crystals-local-systems}
	Let $X$ be a formal scheme adic over $\Spf\O_L$ with adic generic fiber $X_\eta$ over $\Spa(L,\O_L)$.
	\begin{enumerate}
		\item[(1)] There are equivalence of categories
		\begin{align*}
			\Vect(X_{\Prism_L}, \O_\Prism[\tfrac{1}{\mc I}]^\wedge)^{\phi=1}&\simeq \Loc_{\O_L}(X_\eta), \quad\quad\text{and} \\
			D_{\perf}(X_{\Prism_L}, \O_\Prism[\tfrac{1}{\mc I}]^\wedge)^{\phi=1}&\simeq D_{lisse}^b(X_\eta,\O_L).
		\end{align*}
		\item[(2)] Let $\mc M\in D_{\perf}(X_{\Prism_L}, \O_\Prism[\tfrac{1}{\mc I}]^\wedge)^{\phi=1}$ and $T\in D_{lisse}^b(X_\eta,\O_L)$ correspond under the above equivalence. Then there is an isomorphism
		\[R\Gamma(X_{\Prism_L}, \mc M)^{\phi = 1} \cong R\Gamma(X_{\eta,proet}, T).\]
	\end{enumerate}
\end{thm}
Note that if $\pi = 0$ on $X$, then the theorem is trivial: $X_\eta = \Spf 0$ and $\O_\Prism[\tfrac{1}{\mc I}]^\wedge_{(\pi)} = 0$ (since $\mc I$ is the ideal sheaf generated by $\pi$).  
\begin{proof}
	For part (1), we have
	\begin{align*}
		\Vect(X_{\Prism_L}, \O_\Prism[\tfrac{1}{\mc I}]^\wedge)^{\phi=1} &\simeq \lim_{(A,I)\in X_{\Prism_L}^{\perf}} \Mod_{(A,I)}^{\varphi_q,et} \\
		&\simeq \lim_{\substack{\Spf R\rightarrow X \\ R\text{ perfd } \O_L\text{-alg}}} \Loc_{\O_L}(R[1/\pi]) \\
		&\simeq \Loc_{\O_L}(X_\eta)
	\end{align*}
	where the first equivalence is corollary~\ref{cor:F-crystal-perfect-site}, the second is proposition~\ref{prop:F-crystal-over-perfectoid-local-system} and proposition~\ref{prop:perfect-prism-perfectoid-equiv}, and the final equivalence is by v-descent. The same argument works for the derived categories.

	The proof of part (2) is formally identical:
	\begin{align*}
		R\Gamma(X_{\Prism_L}, \mc M)^{\phi = 1}&\cong \lim_{(A,I)\in X_{\Prism_L}^{\perf}} R\Gamma((X/A)_{\Prism_L}, \mc M)^{\phi = 1} \\
		&\cong \lim_{\substack{\Spf R\rightarrow X \\ R\text{ perfd }\O_L\text{-alg}}} R\Gamma(R[\tfrac1\pi]_{proet}, T) \\
		&\cong R\Gamma(X_{\eta,proet}, T)
	\end{align*}
	where $(X/A)_{\Prism_L}$ denotes the relative prismatic site of $(B,J)\in X_{\Prism_L}$ with a map from $(A,I)$ compatible with the maps $\Spf(A/J),\Spf(B/J)\rightarrow X$, and we're now using corollary~\ref{cor:etale-comparison-perfectoid} instead of proposition~\ref{prop:F-crystal-over-perfectoid-local-system}.
\end{proof}

\subsection{Lubin-Tate \'etale \texorpdfstring{$(\varphi_q,\Gamma)$}{(phi\_q, Gamma)}-modules and Laurent \texorpdfstring{$F$}{F}-crystals}\label{ssec:phi-G-modules-F-crystals}

Let $K/L$ be a $p$-adic field. Recall that work of Kisin-Ren \cite{KR} gives the solid equivalences in the following diagram.
\begin{center}
	\begin{tikzcd}
		\Rep_{\Z_p}(G_K)\ar{r}{\cong}\ar[hook]{d} &\Mod_{\A_K}^{\varphi,\Gamma,et}\ar[dashed]{r}{\cong}\ar[hook]{d} &\Mod_{W(K_\infty^\flat)}^{\varphi,\Gamma,et}\ar[hook]{d} \\
		\Rep_{\Z_p}(G_{K_\infty})\ar{r}{\cong} &\Mod_{\A_K}^{\varphi,et}\ar[dashed]{r}{\cong} &\Mod_{W(K_\infty^\flat)}^{\varphi,et}
	\end{tikzcd}
\end{center}
In this section, we show that theorem~\ref{thm:F-crystals-local-systems}(1) specializes to the top row of this diagram when $X = \Spf(\O_{K_\infty})$ and the bottom row when $X = \Spf(\O_K)$. We'll further find that the comparison morphism in theorem~\ref{thm:F-crystals-local-systems}(2) recovers the results on $\varphi_q$-Herr complexes from \cite[theorem~A]{Venjacob-Herr}.

We begin with the case $X = \Spf(\O_{K_\infty})$.
\begin{thm}
	Let $K/L$ be a $p$-adic field.
	\begin{enumerate}
		\item[(1)] There are equivalences of categories
		\[\Mod_{\A_K}^{\varphi_q,et}\simeq \Mod_{W_L(K_\infty^\flat)}^{\varphi_q,et}\simeq \Vect((\O_{K_\infty})_{\Prism_L}, \O_\Prism[\tfrac{1}{\mc I}]^\wedge_{(\pi)})^{\phi = 1}\simeq \Rep_{\O_L}(G_{K_\infty}).\]
		For the derived category, we similarly have
		\[D_{\perf}(\A_K)^{\phi = 1}\simeq D_{\perf}(W_L(K_\infty^\flat))^{\phi = 1}\simeq D_{\perf}((\O_{K_\infty})_{\Prism_L}, \O_\Prism[\tfrac{1}{\mc I}]^\wedge_{(\pi)})^{\phi = 1}\simeq D_{lisse}^b(K_{\infty,proet}, \O_L).\]
		\item[(2)] For $T\in \Rep_{\O_L}(G_{K_\infty})$ corresponding to $M\in \Mod_{\A_K}^{\varphi_q,et}$ or $\Mod_{W_L(K_\infty^\flat)}^{\varphi_q,et}$ under the equivalence from (1), we have that $R\Gamma(K_{\infty,proet},T)$ is isomorphic to the complex
		\[M\stackrel{\phi - 1}{\longrightarrow} M\]
		concentrated in degrees $0$ and $1$.
	\end{enumerate}
\end{thm}
\begin{proof}
	By proposition~\ref{prop:prism-perfection}, $(\O_{K_\infty})_{\Prism_L}^{\perf}$ has an initial object given by $(W_L(\O_{K_\infty}^\flat), \ker\theta)$. Thus by corollary~\ref{cor:F-crystal-perfect-site} we get the equivalence $\Mod_{W_L(K_\infty^\flat)}^{\varphi_q,et}\simeq \Vect((\O_{K_\infty})_{\Prism_L}, \O_\Prism[\tfrac{1}{\mc I}]^\wedge_{(\pi)})^{\phi = 1}$. Then proposition~\ref{prop:perfection-base-change} and theorem~\ref{thm:F-crystals-local-systems} give the first part of part (1). The argument for the derived categories is identical.

	For part (2) note that, viewing $M$ as a complex concentrated in degree $0$, we have
	\[M^{\phi = 1}:= \mathrm{Cone}(\phi_M - 1)[-1] = \left(M\stackrel{\phi - 1}{\longrightarrow} M\right).\]
	Thus by corollary~\ref{cor:perfection-base-change-phi-invariants}, it suffices to prove part (2) for $M\in \Mod_{W_L(K_\infty^\flat)}^{\varphi_q,et}$ corresponding to $T$. Letting $\mc M\in \Vect((\O_{K_\infty})_{\Prism_L}, \O_\Prism[\tfrac{1}{\mc I}]^\wedge_{(\pi)})^{\phi = 1}$ correspond to $T$ and $M$, we have by theorem~\ref{thm:F-crystals-local-systems}(2) that $R\Gamma((\O_{K_\infty})_{\Prism_L},\mc M)^{\phi = 1}\cong R\Gamma(K_{\infty,proet},T)$. Thus it suffices to show that $R\Gamma((\O_{K_\infty})_{\Prism_L},\mc M)\cong M$; this is given by the following lemma.
\end{proof}

\begin{lem}
	If $\mc M\in \Vect((\O_{K_\infty})_{\Prism_L}, \O_\Prism[\tfrac{1}{\mc I}]^\wedge_{(\pi)})^{\phi = 1}$ then 
	\[R\Gamma((\O_{K_\infty})_{\Prism_L},\mc M)\cong \Gamma((\O_{K_\infty})_{\Prism_L}, \mc M).\]
\end{lem}
\begin{proof}
	We want to show that $H^i((\O_{K_\infty})_{\Prism_L},\mc M) = 0$ for $i\ge 1$. Indeed, by derived $\pi$-completeness and derived Nakayama \cite[Tag 0G1U]{stacks-project}, it suffices to show this upon replacing $\mc M$ with $\mc M/\pi$. By corollary~\ref{cor:F-crystal-perfect-site} we can compute cohomology on the site $(\O_{K_\infty})_{\Prism_L}^{\perf}$, which identifies with the category of perfectoid $\O_L$-algebras over $\O_{K_\infty}$ by proposition~\ref{prop:perfect-prism-perfectoid-equiv}. Under this identification, $\mc M/\pi$ is the sheaf which sends a perfectoid $\O_L$-algebra $S$ over $\O_{K_\infty}$ to
	\[\mc M(\Ainf(S),\ker\theta)/\pi = S[\tfrac1\pi]^\flat\otimes_{K_\infty^\flat}\mc M(\Ainf(\O_{K_\infty}),\ker\theta)/\pi.\]
	Thus it suffices to show that the sheaf which sends a perfectoid $\O_L$-algebra $S$ over $\O_{K_\infty}$ to $S[\tfrac1\pi]^\flat$ has vanishing higher cohomology. But, via the tilting equivalence, this is just the basic fact about Galois cohomology that $H^i(K_{\infty}^\flat, \overline{K}^\flat) = 0$ for $i \ge 1$.
\end{proof}

Naively, we might hope to deduce the corresponding result for $X = \Spf\O_K$ by descent along $\Spf\O_{K_\infty}\rightarrow \Spf\O_K$. However, instead of using this angle of attack, we will use a more delicate descent argument along the \v{C}ech nerve $(W_L(\O_{K_\infty}^\flat), \ker\theta)^\bullet$ in the perfect prismatic site $(\O_K)_{\Prism_L}^{\perf}$. This approach, which is the same as the one in \cite[proof of theorem~5.2]{Wu}, allows us to recover a Laurent $F$-crystal $\mc M$ over $(\O_K)_{\Prism_L}$ from the data of $\mc M(W_L(\O_{K_\infty}^\flat), \ker\theta)$ and a semilinear action of $\Aut_{(\O_K)_{\Prism_L}}(W_L(\O_{K_\infty}^\flat), \ker\theta)\cong \Gamma_K$ (by proposition~\ref{prop:Ainf-automorphisms}).

\begin{lem}\label{lem:Ainf-cover-of-final-object}
	$(\Ainf(\O_{K_\infty}), \ker\theta)$ is a cover of the final object of the topos $\mathrm{Shv}((\O_K)_{\Prism_L})$.
\end{lem}
\begin{proof}
	We want to show that for any $(A,I)\in (\O_K)_{\Prism_L}$, there is a cover $(B,J)$ of $(A,I)$ with a map $(\Ainf(\O_{K_\infty}), \ker\theta)\rightarrow (B,J)$. Let $(\Ainf(R),\ker\theta) = (A,I)_{\perf}$, using proposition~\ref{prop:perfect-prism-perfectoid-equiv}. As $\O_K\rightarrow \O_{K_\infty}$ is $\pi$-completely faithfully flat, so is $R\rightarrow S:=R\hat\otimes_{\O_K}^L\O_{K_\infty}$, where $S$ is the derived $\pi$-completion of the derived tensor product. Using the same argument as in \cite[IV, proposition~2.11]{bhattnotes}, we have that $S$ is a perfectoid $\O_L$-algebra. Thus by proposition~\ref{prop:prism-perfection} and lemma~\ref{lem:pi-faithfully-flat-pi-I-faithfully-flat}, we have that the composite
	\[(A,I)\rightarrow (\Ainf(R),\ker\theta)\rightarrow (\Ainf(S),\ker\theta)\]
	is a cover in $(\O_K)_{\Prism_L}$. But also from the map $\O_{K_\infty}\rightarrow S$, we get a morphism $(\Ainf(\O_{K_\infty}),\ker\theta)\rightarrow (\Ainf(S),\ker\theta)$ as desired.
\end{proof}
\begin{lem}\label{lem:self-product-computation}
	Let $n\ge 1$ and let 
	\[(B,J) = (\Ainf(\O_{K_\infty}),\ker\theta)^{(n + 1)} := (\Ainf(\O_{K_\infty}),\ker\theta) \times \dots \times (\Ainf(\O_{K_\infty}),\ker\theta)\]
	be the $(n+1)$-times iterated self-product in $(\O_K)_{\Prism_L}^{\perf}$. Then
	\begin{align*} 
		B &= \mathrm{Hom}_{\mathrm{cont}}(\Gamma_K,W_L(\O_{K_\infty}^\flat))\quad\quad\text{and} \\
		B[\tfrac{1}{J}]^\wedge_{(\pi)} &= \mathrm{Hom}_{\mathrm{cont}}(\Gamma_K, W_L(K_\infty^\flat)).
	\end{align*}
\end{lem}
\begin{proof}
	We first compute $B$. By proposition~\ref{prop:perfect-prism-perfectoid-equiv}, we are interested in the self-product of $\O_{K_\infty}$ in the category of perfectoid $\O_L$-algebras over $\O_K$. To compute this, let $U = \plim\Spa(K_m,\O_{K_m})$ be the element of the pro-\'etale site $X_{proet}$ for $X = \Spa(K,\O_K)$. By \cite[lemma~4.10]{S}, the self-product we are looking for can be computed as $H^0(U^{(n+1)},\hat\O_X^+)$ where $U^{(n+1)} = U\times_X\dots\times_X U$. As $U\rightarrow X$ is Galois with Galois group $\Gamma_K$, we have $U^{(n + 1)} = U\times \Gamma_K^n$ where $\Gamma_K^n$ is viewed in $X_{proet}$ as a profinite set with trivial Galois action (cf. \cite[proof of lemma~5.6]{S}). But then \cite[theorem~4.9]{S} and \cite[lemma~3.16]{S} imply that
	\[H^0(U\times \Gamma_K^n, \hat\O_X^+) = \mathrm{Hom}_{\mathrm{cont}}(\Gamma_K^n, H^0(U,\hat\O_X^+)) = \mathrm{Hom}_{\mathrm{cont}}(\Gamma_K^n, \O_{K_\infty}).\]
	It is easy to verify that tilting and taking $W_L(-)$ commutes with $\mathrm{Hom}_{\mathrm{cont}}(\Gamma_K^n, -)$, giving the first part of the result.

	Since $B/J$ is a perfectoid $\O_L$-algebra, we have $B[\tfrac{1}{J}]^\wedge_{(\pi)} = W_L(B/J[\tfrac{1}{\pi}]^\flat)$. We thus have
	\[B[\tfrac{1}{J}]^\wedge_{(\pi)} = W_L\left(\mathrm{Hom}_{\mathrm{cont}}(\Gamma_K^n, \O_{K_\infty})[\tfrac{1}{\pi}]^\flat\right) = \mathrm{Hom}_{\mathrm{cont}}(\Gamma_K^n, W_L(K_\infty^\flat))\]
	as desired.
\end{proof}

\begin{thm}\label{thm:phi-G-modules-F-crystals}
	Let $K/L$ be a $p$-adic field.
	\begin{enumerate}
		\item[(1)] There are equivalences of categories
		\[\Mod_{\A_K}^{\varphi_q,\Gamma_K, et}\simeq \Mod_{W_L(K_\infty^\flat)}^{\varphi_q,\Gamma_K, et} \simeq \Vect((\O_K)_{\Prism_L}, \O_\Prism[\tfrac{1}{\mc I}]^\wedge_{(\pi)})^{\phi=1} \simeq  \Rep_{\O_L}(G_K)\]
		and similarly for the corresponding derived categories.
		\item[(2)] Let $T\in \Rep_{\O_L}(G_K)$ correspond to $M\in \Mod_{\A_K}^{\varphi_q,\Gamma_K,et}$ or $\Mod_{W_L(K_\infty^\flat)}^{\varphi_q,\Gamma_K,et}$ under the equivalence from (1). Let $C^\bullet_{\mathrm{cont}}(\Gamma_K,M)$ denote the continuous cochain complex of $\Gamma_K$ with values in $M$. Then $R\Gamma(K_{proet}, T)$ is isomorphic to $C^\bullet_{\mathrm{cont}}(\Gamma_K,M)^{\phi = 1}$. 
	\end{enumerate}
\end{thm}
\begin{proof}[Proof of theorem~\ref{thm:phi-G-modules-F-crystals}]
	The first and last equivalences in the theorem follow from proposition~\ref{prop:perfection-base-change} and theorem~\ref{thm:F-crystals-local-systems}, so we focus on the equivalence $\Mod_{W_L(K_\infty^\flat)}^{\phi,\Gamma_K,et}\simeq \Vect((O_K)_{\Prism_L}, \O_{\Prism}[\tfrac{1}{I}]^\wedge_{(\pi)})^{\phi = 1}$. Since $(W_L(\O_{K_\infty}^\flat),\ker\theta)$ is a cover of the final object $*$ of $\mathrm{Shv}((\O_K)_{\Prism_L})$ by lemma~\ref{lem:Ainf-cover-of-final-object}, we have that
	\[\Vect((\O_K)_{\Prism_L}, \O_\Prism[\tfrac{1}{I}]^\wedge_{(\pi)})^{\phi = 1}\simeq \plim\left(
	\begin{tikzcd}
		\Mod_{(W_L(\O_{K_\infty}^\flat),\ker\theta)}^{\phi,et}\ar[shift left]{r}\ar[shift right]{r} & \Mod_{(W_L(\O_{K_\infty}^\flat),\ker\theta)^{(2)}}^{\phi,et}\ar[shift left=2]{r}\ar{r}\ar[shift right=2]{r} & \dots
	\end{tikzcd}
	\right)\]
	where $(W_L(\O_{K_\infty}^\flat),\ker\theta)^{(2)} := (W_L(\O_{K_\infty}^\flat),\ker\theta)\times_*(W_L(\O_{K_\infty}^\flat),\ker\theta)$ denotes the self product in $(\O_K)_{\Prism_L}^{\perf}$ (here we have used a general fact about recoving a vector bundle from its value on the \v{C}ech nerve of a cover of the final object; see \cite[footnote~10]{BS} or \cite[\secsymb3]{Wu} for more details). By lemma~\ref{lem:self-product-computation} we then get
	\[\Vect((\O_K)_{\Prism_L}, \O_\Prism[\tfrac{1}{I}]^\wedge_{(\pi)})^{\phi = 1}\simeq \plim\left(
	\begin{tikzcd}
		\Mod_{W_L(K_\infty^\flat)}^{\phi,et}\ar[shift left]{r}\ar[shift right]{r} & \Mod_{\mathrm{Hom}_{\mathrm{cont}}(\Gamma_K, W_L(K_\infty^\flat))}^{\phi,et}\ar[shift left=2]{r}\ar{r}\ar[shift right=2]{r} & \dots
	\end{tikzcd}
	\right).\]
	By the same argument as for usual Galois descent, this identifies $\Vect((\O_K)_{\Prism_L}, \O_\Prism[\tfrac{1}{I}]^\wedge_{(\pi)})^{\phi = 1}$ with the category of \'etale $\varphi_q$-modules over $W_L(K_\infty^\flat)$ with a semilinear action of $\Gamma_K$ which also commutes with $\phi$. But this is exactly the definition of the category $\Mod_{W_L(K_\infty^\flat)}^{\phi,\Gamma_K,et}$, giving part (1).

	For part (2), we can again focus on the case $M\in \Mod_{W_L(K_\infty^\flat)}^{\varphi_q,\Gamma_K,et}$ by corollary~\ref{cor:perfection-base-change-phi-invariants}. For $\mc M\in \Vect((\O_K)_{\Prism_L}, \O_\Prism[\tfrac{1}{\mc I}]^\wedge_{(\pi)})^{\phi = 1}$ corresponding to $T$ and $M$, we get by the same computation as above that $R\Gamma((\O_K)_{\Prism_L}, \mc M)\simeq C^\bullet_{\mathrm{cont}}(\Gamma_K, M)$. We then conclude by theorem~\ref{thm:F-crystals-local-systems}(2).
\end{proof}

\bibliographystyle{habbrv}
\bibliography{bibliography}


\end{document}